\documentclass[11pt]{article}

\usepackage{dutchcal}
\usepackage[english]{babel}
\newtheorem{theorem}{Theorem}
\newtheorem{lemma}{Lemma}
\newtheorem{definition}{Definition}
\newtheorem{proof}{Proof}
\usepackage{stmaryrd}
\usepackage{enumerate}
\usepackage{algorithm}
\usepackage[noend]{algpseudocode}
\usepackage{calrsfs}
\usepackage{graphicx}
\usepackage{amssymb}
\usepackage{array}
\usepackage{url}
\usepackage{amsmath}
\usepackage{epstopdf}
\usepackage{epsfig}
\usepackage{subcaption}
\usepackage{pstricks}
\usepackage{fancyheadings}
\usepackage{pgfplots}
\usepackage{physics}
\usepackage{stackengine}
\usepackage[margin=1.0in]{geometry}
\DeclareMathAlphabet{\pazocal}{OMS}{zplm}{m}{n}

\usepackage{graphicx}

\usepackage[framemethod=tikz]{mdframed}
\usepackage{lipsum}

\definecolor{mycolor}{rgb}{0.122, 0.435, 0.698}
\definecolor{blue-green}{rgb}{0.0, 0.87, 0.87}
\definecolor{royalblue}{rgb}{0.01, 0.28, 1.0}
\definecolor{bluet}{rgb}{0.1, 0.1, 0.8}

\newmdenv[innerlinewidth=1pt, roundcorner=0.5pt,linecolor=mycolor,innerleftmargin=0.5pt,
innerrightmargin=0.5pt,innertopmargin=0.4pt,innerbottommargin=0.4pt]{mybox}

\DeclareMathAlphabet{\pazocal}{OMS}{zplm}{m}{n}

\usepackage{bm}

\newcommand{\ba}{\begin{array}}
\newcommand{\ea}{\end{array}}
\newcommand{\be}{\begin{equation}}
\newcommand{\ee}{\end{equation}}
\newcommand{\ben}{\begin{equation*}}
\newcommand{\een}{\end{equation*}}
\newcommand{\bd}{\begin{displaymath}}
\newcommand{\ed}{\end{displaymath}}
\newcommand{\bi}{\begin{itemize}}
\newcommand{\ei}{\end{itemize}}
\newcommand{\bn}{\begin{enumerate}}
\newcommand{\en}{\end{enumerate}}

\newcommand{\f}{\frac}

\newcommand{\ve}{\varepsilon}

\newcommand{\bs}{\boldsymbol}
\newcommand{\mb}{\mathbf}

\newcommand{\bp}{\beta}
\newcommand{\mc}{\mathbcal}
\newcommand{\tauw}{\tau_w}
\newcommand{\tauu}{\tau_u}
\newcommand{\sigmau}{\sigma_u}
\newcommand{\sigmaw}{\sigma_w}
\newcommand{\paratr}{\beta_{inv}}
\newcommand{\vertiii}[1]{{\left\vert\kern-0.25ex\left\vert\kern-0.25ex\left\vert #1 
    \right\vert\kern-0.25ex\right\vert\kern-0.25ex\right\vert}}
    
\newtheorem{remark}{Remark}
\newtheorem{corollary}{Corollary}

\begin{document}

\title{A finite difference - discontinuous Galerkin method for the wave equation in second order form}
\author{Siyang Wang\thanks{Department of Mathematics and Mathematical Statistics, Umeå University, Umeå, Sweden. Email: siyang.wang@umu.se}  \and Gunilla Kreiss\thanks{Division of Scientific Computing, Department
    of Information Technology, Uppsala University, Uppsala, Sweden.}}
    \date{}
\maketitle

\begin{abstract}
 We develop a hybrid spatial discretization for the wave equation in second order form, based on high-order accurate finite difference methods and discontinuous Galerkin methods. The hybridization combines computational efficiency of finite difference methods on Cartesian grids and geometrical flexibility of discontinuous Galerkin methods on unstructured meshes. The two spatial discretizations are coupled by a penalty technique at the interface such that the overall semidiscretization satisfies a discrete energy estimate to ensure stability. In addition, optimal convergence is obtained in the sense that when combining a fourth order finite difference method with a discontinuous Galerkin method using third order local polynomials, the overall convergence rate is fourth order. Furthermore, we use a novel approach to derive an error estimate for the semidiscretization by combining the energy method and the normal mode analysis for a corresponding one dimensional model problem. The stability and accuracy analysis are verified in numerical experiments. 
\end{abstract}


\noindent \textbf{AMS}: 65M06, 65M12

\section{Introduction}
Second order hyperbolic partial differential equations describe wave-dominated problems, for example the acoustic wave equation, the elastic wave equation and Einstein's equations of general relativity. In realistic models, waves propagate over long time in large domains with heterogeneous material properties and complex geometries. As a result, analytical solutions can  generally not be derived. Numerical simulation is a powerful alternative to seek an approximated solution to the governing equations. For time-dependent problems, it is important to use stable numerical methods that do not allow unphysical growth in the numerical solution. In addition, by the classical dispersion analysis \cite{Hagstrom2012,Kreiss1972}, high-order accurate numerical methods are more computationally efficient than low-order methods when the solution is sufficiently smooth. Over the years, there has been extensive work on stable and high-order numerical methods for wave propagation problems.

The finite difference (FD) method is conceptually simple, computationally efficient and easy to implement. Traditionally, it was challenging to derive stable and high-order FD discretizations for hyperbolic problems. This challenge has partly been overcome by using FD stencils with a summation-by-parts (SBP) property \cite{Kreiss1974}, in combination with the simultaneous-approximation-term (SAT) technique \cite{Carpenter1994} to impose boundary conditions. The integration-by-parts principle is the key ingredient to derive continuous energy estimates for the PDEs. The SBP-SAT methodology mimics the integration-by-parts principle for a discrete energy estimate to ensure that the semidiscretization is stable. The relation between the SBP-SAT FD method and the discontinuous Galerkin spectral element method is investigated in \cite{Gassner2013}. 

The FD method in its basic form is only applicable to problems on rectangular-shaped domains. For other shapes, a curvilinear grid based on coordinate transformation is used to resolve geometrical features \cite{Svard2004}. In general, the computational domain cannot be easily mapped to a reference domain. In this case, we decompose the computational domain into subdomains and use a multiblock FD approach. 
The multiblock SBP-SAT methods on curvilinear grid have been derived for the wave equation \cite{Virta2014} and the elastic wave equation \cite{Duru2014V} in second order form. This approach works well on nearly Cartesian grids but is not suitable in many realistic models with complex geometry, because it is difficult to find a smooth coordinate transformation. 

Recently, there have been efforts in hybridizing the FD discretization with a Galerkin method on unstructured meshes so that the overall discretization is both computationally efficient and geometrically flexible. The main difficulty originates from the fact that the two discretizations have different discrete $l^2$ inner product. This scenario also occurs at an FD-FD  discretization with different grid sizes, i.e. nonconforming grid interfaces. For the wave equation in first order form, SBP-preserving interpolation operators are constructed in \cite{Mattsson2010} for an FD-FD nonconforming interface with grid size ratio 1:2. With an SBP operator of interior order $2p$, the observed convergence rate in numerical experiments is $p+1$, which is the same as a multiblock FD with only conforming grid interfaces.  
In \cite{Kozdon2016}, the SBP FD is coupled with the discontinuous Galerkin (DG) method by using a  projection technique that preserves the SBP property and the semidiscretization satisfies an energy estimate. With an SBP operator of interior order $2p$ and the DG method based on local polynomials of degree $p$, the observed convergence rate in numerical experiments is $p+1$. There has also been important work on the hybridization of the SBP FD discretization with the finite element method for the isotropic elastic wave equation \cite{Gao2019} and the conservation law \cite{Dao2022}, with a focus on stability rather than accuracy. 

In this paper, we consider the wave equation in second order form. Comparing to first order form, solving the wave equations in second order form has advantages. There are fewer unknown variables, thus requiring less computation and memory storage. In addition, when imposing the boundary and interface conditions properly,  the SBP FD discretization based on operators of interior order $2p$ can converge to order $p+2$ , i.e. one order higher than solving the same equation in first order form. However, it is challenging to solve the wave equations in second order form from both stability and accuracy aspects. A  generalization of the interpolation technique from \cite{Mattsson2010} to the wave equation in second order form converges only to suboptimal order $p+1$. For stability, an additional norm-contraction constraint on the interpolation operators is required. This additional constraint is removed by using a new SAT technique \cite{Wang2018}, which does not simultaneously improve the accuracy property. In \cite{Almquist2019}, the optimal convergence rate  $p+2$ is recovered by using two pairs of order-preserving interpolation operators. 

The first contribution of this paper is an FD-DG spatial discretization for the wave equation in two space dimension in second order form. We construct novel projection operators to combine the SBP FD discretization with the symmetric interior penalty discontinuous Galerkin (IPDG) method \cite{Grote2006}. The overall discretization satisfies a discrete energy estimate to guarantee stability. In addition, the FD-DG discretization converges to the optimal order in the sense that with SBP operators of interior order four and the IPDG based on local polynomials of degree three, the observed convergence rate is four. 

Our second contribution is a new framework for the accuracy analysis of the FD-DG discretization. A priori error estimates for the DG discretization are often derived by the energy method using special projection operators and approximation theory \cite{Hesthaven2008}, whereas sharp error estimates for the FD discretization is derived by the normal mode analysis in Laplace space \cite{Gustafsson2008,Gustafsson2013}. Though both are well-established, they are two distinct approaches. To analyze the accuracy of the FD-DG discretization, we consider the wave equation in one space dimension and cast the DG scheme into matrix form, and realize its components as difference stencils. It is well-known that the resulting DG truncation error indicates a suboptimal convergence rate. By a careful analysis of the truncation error in the discrete norm associated with the DG discretization, we obtain sharp error estimates by the energy method  for the DG discretization. After that, we combine it with the normal mode analysis for the FD-DG interface treatment and obtain an optimal convergence rate for the overall discretization. 

The rest of the paper is organized as follows. In Sec. 2, we introduce an FD-DG spatial discretization for the wave equation in one space dimension.  After that, we present our novel approach for deriving an apriori error estimate for the hybridization. In Sec. 3, we start with projection operators that are used in the numerical scheme for the wave equation in two space dimension. We then analyze the stability property of the overall discretization by deriving a discrete energy estimate. Numerical examples are presented in Sec. 4 to verify the theoretical results. In the end, we draw conclusion in Sec. 5.

\section{Spatial discretization in 1D and error analysis}
In this section, we start by introducing the concept of SBP and its important properties, and deriving an FD-DG spatial discretization of the wave equation in one space dimension.  After that, we present a novel approach for accuracy analysis and derive an a priori error estimate for the FD-DG semidiscretization. 

\subsection{Summation-by-parts finite difference operators}
Consider a bounded interval $I$ that is discretized by a uniform grid $x_i,\ i=1,2,\cdots,n$ with grid spacing $h$. Let $f,g\in C^{\infty}(I)$ and define the grid functions $f_i = f(x_i),g_i = g(x_i)$, and vectors
\[
\mb{f} = [f_1, f_2, \cdots, f_n]^T,\quad \mb{g} = [g_1, g_2, \cdots, g_n]^T.
\]
We also define the standard $L^2$ inner product $(f,g)_I=\int_I fgdx$, and a discrete $l^2$ norm $\|\mb{f}\|=\sqrt{h\sum_{i=1}^n |f_i|^2}$.

Next, we consider the finite difference approximation of the second derivative, $D\approx \f{d^2}{dx^2}$. The SBP property of $D$ is defined as follows \cite{Mattsson2004}. 
\begin{definition}[second derivative SBP property]\label{defD2}
The finite difference operator $D\approx \f{d^2}{dx^2}$ is a second derivative SBP operator if it can be written as  
\be\label{D2}
D=H^{-1}(-A+\mb{e}_n\mb{d}_n^T
-\mb{e}_1\mb{d}_1^T),
\ee 
where $\mb{e}_n=[0,0,\cdots,0,1]^T$ and $\mb{e}_1=[1,0,\cdots,0]^T$. 
The first derivative approximations are $\mb{d}_1^T \mb{f}\approx \f{df}{dx}(x_1)$ and $\mb{d}_n^T \mb{f}\approx \f{df}{dx}(x_n)$. The operator $H$ is symmetric positive definite, and $A$ is symmetric positive semidefinite. 
\end{definition}

The operator $H$ defines a discrete inner product and norm, and is also a quadrature  \cite{Hicken2013}. Similarly, the operator $A$ defines a discrete semi-norm. They satisfy the relations,
\[
\mb{f}^T H \mb{g}\approx \int_{x_1}^{x_n} fgdx,\quad  \mb{f}^T A \mb{g}\approx \int_{x_1}^{x_n} \f{df}{dx}\f{dg}{dx}dx.
\]
We recognize $H$ and $A$ as the mass and stiffness matrix for a Galerkin method. 

In the interior, the SBP operators $D$ are based on standard central finite difference stencils with truncation error $\mathcal{O}(h^{2p})$.  On a few grid points near boundaries, one-sided stencils are used to satisfy the SBP property. When $H$ is diagonal, the truncation error of the one-sided boundary stencil can at best be $\mathcal{O}(h^{p})$. The truncation error of the first derivative approximation at the boundaries is $\mathcal{O}(h^{p+1})$. We denote the order of accuracy of $D$ as $(2p,p)$.  
The SBP property of \eqref{D2} can also be written as 
\[
\mb{g}^T HD\mb{f}=-\mb{g}^T A \mb{f} + \mb{g}^T \mb{e}_n\mb{d}_n^T\mb{f}- \mb{g}^T \mb{e}_1\mb{d}_1^T\mb{f},
\]
which is a discrete analogue of the integration-by-parts formula,
\[
\int_{x_1}^{x_n} g f_{xx} dx = -\int_{x_1}^{x_n} g_xf_x dx +g(x_n)f_x(x_n)-g(x_1)f_x(x_1). 
\]

A so-called \textit{borrowing technique} of the SBP operator $D$ is important for proving stability for certain problems, such as the wave equation with Dirichlet boundary conditions \cite{Mattsson2009} and material interface conditions \cite{Mattsson2008}. It is also used to derive an energy estimate for a dual-consistent discretization of the heat equation \cite{Eriksson2018}. The \textit{borrowing capacity} for the \textit{borrowing technique} is defined as follows. 
\begin{definition}[borrowing capacity]\label{lemB}
The borrowing capacity is the maximum value of $\bp>0$ such that 
\[
\tilde A=A - \bp h (\mb{d}_1\mb{d}_1^T + \mb{d}_n\mb{d}_n^T)  
\]
  is symmetric positive semidefinite. Here, $h$ is the grid spacing, $\mb{d}_1$ and $\mb{d}_n$ are the same first derivative operators as in \eqref{D2}. 
  \end{definition}
\begin{remark}
The borrowing capacity depends on the order of accuracy of the SBP operator but does not depend on $h$. For the precise values of the borrowing capacity, see   \cite{Eriksson2021,Mattsson2008,Mattsson2009}. 
The borrowing technique is a finite difference analogue to using the inverse inequality to derive estimates for finite element methods. To see this relation, we write 
\[
\mb{f}^TA\mb{f} - \bp h \mb{f}^T(\mb{d}_1\mb{d}_1^T + \mb{d}_n\mb{d}_n^T) \mb{f} = \mb{f}^T\tilde A\mb{f}\geq 0,
\]
which leads to 
\[
\mb{f}^TA\mb{f} \geq \beta h ((\mb{d}_1^T \mb{f})^2+(\mb{d}_n^T \mb{f})^2).
\]
Recalling $\mb{f}^TA\mb{f}\approx \int_{x_1}^{x_n} (\f{df}{dx})^2dx$, $\mb{d}_1^T \mb{f}\approx \f{df}{dx}(x_1)$, and $\mb{d}_n^T \mb{f}\approx \f{df}{dx}(x_n)$ the above relation is a discrete analogue of the inverse inequality \cite{Brenner2007}. 
\end{remark}

\subsection{An FD-DG discretization in 1D}
An SBP operator only approximates a derivative but does not impose any boundary condition. When solving an initial-boundary-value problem, the SAT technique is often used to impose boundary and interface conditions weakly. The main idea of SAT is to add penalty terms in the semidiscretization such that a discrete energy estimate can be obtained. For accuracy, it is important that the penalty terms converge to zero as the mesh size goes to zero. The SBP-SAT discretization for the wave equation in second order form was derived for various boundary conditions \cite{Appelo2007,Mattsson2009,Mattsson2004} and material interface conditions \cite{Mattsson2008}.

In the IPDG method \cite{Grote2006}, boundary and material interface conditions are naturally imposed by using numerical fluxes. In the following, we use the wave equation in one space dimension as the model problem, and derive a  stable FD-DG semidiscretization. In this case, the interface between the two semidiscretizations is only a point in space and the numerical treatment does not involve the difficulties for higher dimensional problems. Nonetheless, the scheme and stability analysis for the one dimensional model problem demonstrate the penalty technique to combine the FD and DG semidiscretizations and prepare for the accuracy analysis afterwards. 

For the analysis, we consider 
\[
U_{tt} = U_{xx},\quad x\in (-\infty, \infty),\quad t\in (0,T],
\]
with smooth initial conditions with bounded support. We discretize the equation in space by the SBP FD method  in $x\in (-\infty, 0)$, and the IPDG in $x\in (0,\infty)$. At the FD-DG interface at $x=0$, we impose the interface conditions $U(0^-,t)=U(0^+,t)$ and $U_x(0^-,t)=U_x(0^+,t)$ weakly. 

\begin{figure}
\centering
\includegraphics[trim={0 5cm 0 4cm},clip,width=0.6\textwidth]{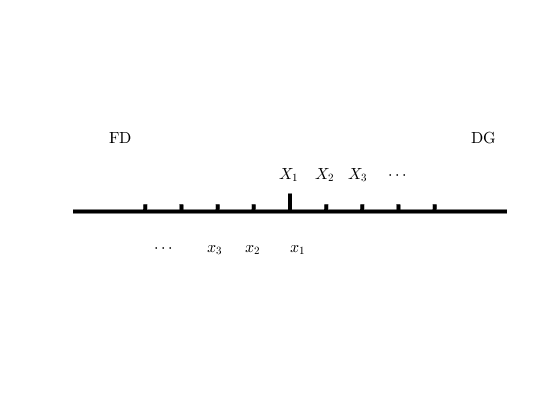}
\caption{An FD grid and DG elements in one space dimension. }
\label{Domain1D}
\end{figure}

We discretize the FD domain $(-\infty,0)$ by a uniform grid $x_j = -(j-1)h$, where $j=1,2,3,\cdots$ and $h$ is the grid spacing. In the DG domain, we partition $(0,\infty)$ into disjoint elements $I_j=(X_j, X_{j+1})$ with $j=1,2,3,\cdots$. For simplicity, we assume that the elements have equal length such that $X_j = (j-1)h,\ j=1,2,3,\cdots$. We note that the points $x_1$ and $X_1$ coincide at the FD-DG interface, see Figure \ref{Domain1D}. We also note that the degrees of freedom (DOFs) are duplicated on the inter-element interfaces $X_j,\ j=2,3,\cdots$, on the DG side. 

\subsection{Stability of the FD-DG discretization in 1D}
The FD discretization can be written as
\begin{align}
\begin{split}
\mc{w}_{tt} = &H^{-1} (-A+\mb{e}_n\mb{d}_n^T) \mc{w}\\
& -\frac{1}{2}H^{-1} \mb{e}_n (\mb{d}_n^T \mc{w} -u_{x\Gamma}^{(1)}) +\frac{1}{2}H^{-1}\mb{d}_n( \mb{e}_n^T  \mc{w}-u_{\Gamma}^{(1)})-\frac{\tau}{h}H^{-1} \mb{e}_n (\mb{e}_n^T\mc{w}-u_{\Gamma}^{(1)}),
\end{split}
\label{FD1d}
\end{align}
where $\mc{w}=[w_1, w_2,\cdots]^T$ is the finite difference solution, $w_j\approx U(x_j, t), j=1,2,\cdots$. On the right-hand side, the first term is the approximation of $U_{xx}$, and the last three terms impose weakly the interface conditions. More precisely, the second term imposes continuity of $U_x$, and the third and fourth terms impose weakly continuity of $U$. The terms $u_{\Gamma}^{(1)}$ and $u_{x\Gamma}^{(1)}$ are the DG solution and its derivative at the interface, i.e. $u_{\Gamma}^{(1)}=u^{(1)}(X_1,t)$ and $u_{x\Gamma}^{(1)}=u_x^{(1)}(X_1,t)$. We note that \eqref{FD1d} is a generalization of the SBP-SAT scheme for the 1D wave equation with a material interface \cite{Mattsson2008}. 

For the DG solution, for every fixed time we seek solution in the following space
\begin{equation}\label{Vh}
V_h^k = \{  v: v|_{I_j}\in\mathcal{P}^k(I_j),\ j=1,2,\cdots  \},
\end{equation}
where $\mathcal{P}^k(I_j)$ denotes the space of polynomials of degree at most $k$ in $I_j$. The DG discretization reads: for any fixed $t$, find $u \in V_h^k$  such that 
\begin{align}
\begin{split}
(u_{tt}^{(j)}, \phi^{(j)})_{I_j} =& -(u_x^{(j)}, \phi_x^{(j)})_{I_j} + (u_x^{(j)}, \phi^{(j)})_{X_{j+1}}-(u_x^{(j)}, \phi^{(j)})_{X_{j}}\\
 &-\frac{1}{2}(u_x^{(j)}-u_x^{(j+1)}, \phi^{(j)})_{X_{j+1}}+\frac{1}{2}(u^{(j)}-u^{(j+1)},\phi_x^{(j)})_{X_{j+1}}-\frac{\tau}{h}(u^{(j)}-u^{(j+1)}, \phi^{(j)})_{X_{j+1}}\\
 &+\frac{1}{2}(u_x^{(j)}-u_x^{(j-1)}, \phi^{(j)})_{X_{j}}-\frac{1}{2}(u^{(j)}-u^{(j-1)}, \phi_x^{(j)})_{X_{j}}-\frac{\tau}{h}(u^{(j)}-u^{(j-1)}, \phi^{(j)})_{X_{j}},
 \end{split}
 \label{DG1d}
\end{align}
for all $\phi^{(j)}\in\mathcal{P}^k(I_j)$ and $j=1,2,\cdots.$ In \eqref{DG1d}, the first line is obtained by using the integration-by-parts formula. The three terms on the second line of \eqref{DG1d} are numerical fluxes for element $I_j$ and $I_{j+1}$. Similarly, the three terms on the third line of \eqref{DG1d} are numerical fluxes for element $I_j$ and $I_{j-1}$. When $j=1$, the values of $u^{(0)}(X_1,t)$ and $u_x^{(0)}(X_1,t)$ are obtained from the FD solutions. More precisely, we define  $u^{(0)}(X_1,t)=\mb{d}_n^T \mc{w}$ and $u_x^{(0)}(X_1,t)=\mb{e}_n^T \mc{w}$. 

For proving stability of the semidiscretization \eqref{FD1d}-\eqref{DG1d}, we need to use the standard inverse inequality \cite{Brenner2007} formulated in the following lemma.
\begin{lemma}[Inverse inequality] \label{lemma_inv}
For any function $u^{(j)}\in \mathcal{P}^k(I_j)$, $j=1,2,\cdots$, there exists a constant $\paratr$ such that 
\[
(u_x^{(j)}, u_x^{(j)})_{I_j} \leq \paratr h^{-1} ((u_x^{(j)}|_{X_j})^2+(u_x^{(j)}|_{X_{j+1}})^2),\quad j=1,2,\cdots,
\]
where $\paratr>0$ is a constant that depends on $k$ but not $h$. 
\end{lemma}
We state the stability property of the FD-DG discretization \eqref{FD1d}-\eqref{DG1d} in the following theorem, and  prove it by deriving a discrete energy estimate. 

\begin{theorem}[Stability in 1D]\label{ec1d}
If $\tau\geq \frac{1}{2\tilde\beta}$ in the FD-DG semidiscretization \eqref{FD1d}-\eqref{DG1d} with $\tilde\beta=\max(\beta, \paratr)$, then 
\begin{align*}
 E_h :=\ & \mc{w}_t^TH\mc{w}_t +\sum_{j=1}^{\infty} (u_{t}^{(j)}, u_{t}^{(j)})_{I_j} + \mc{w}^T A\mc{w} + \sum_{j=1}^{\infty}  (u_x^{(j)}, u_{x}^{(j)})_{I_j} \\
&-(\mb{e}_n^T \mc{w}-u_{\Gamma}^{(1)})(\mb{d}_n^T\mc{w}+u_{x\Gamma}^{(1)})+\frac{\tau}{h}(\mb{e}_n^T\mc{w}-u_{\Gamma}^{(1)})^2\\
&+\sum_{j=2}^{\infty}(-(u^{(j-1)}-u^{(j)})(u_x^{(j-1)}+u_x^{(j)})+\frac{\tau}{h}(u^{(j-1)}-u^{(j)})^2)|_{X_j}
\end{align*}
defines a semidiscrete energy $E_h\geq 0$ and satisfies 
$
\frac{d}{dt}E_h = 0.
$ 
\end{theorem}

\begin{proof}
We multiply \eqref{FD1d} by $\mc{w}_t^TH$ from the left, and obtain
\begin{align*}
\mc{w}_t^TH\mc{w}_{tt} = -\mc{w}_t^T A\mc{w}  + \frac{1}{2}\mc{w}_t^T\mb{e}_n\mb{d}_n^T\mc{w} + \frac{1}{2}\mc{w}_t^T\mb{d}_n\mb{e}_n^T\mc{w}+\frac{1}{2}\mc{w}_t^T\mb{e}_n u_{x\Gamma}^{(1)} -\frac{1}{2}\mc{w}_t^T\mb{d}_n u_{\Gamma}^{(1)}   -\frac{\tau}{h}\mc{w}_t^T\mb{e}_n\mb{e}_n^T\mc{w}+\frac{\tau}{h}\mc{w}_t^T\mb{e}_n u_{\Gamma}^{(1)}. 
\end{align*}
In the DG part, we choose $\phi^{(j)}$ to be $u_t^{(j)}$ in \eqref{DG1d}, and sum in $j$ to obtain
\small
\begin{align*}
\sum_{j=1}^{\infty} (u_{tt}^{(j)}, u_{t}^{(j)})_{I_j} =& -\sum_{j=1}^{\infty}  (u_x^{(j)}, u_{tx}^{(j)})_{I_j} \\
&+ \sum_{j=1}^{\infty} \left[\frac{1}{2}u_x^{(j)}u_t^{(j)}+ \frac{1}{2}u^{(j)}u_{tx}^{(j)}+ \frac{1}{2}u_x^{(j+1)}u_t^{(j)}- \frac{1}{2}u^{(j+1)}u_{tx}^{(j)}- \frac{\tau}{h}u^{(j)}u_t^{(j)}+\frac{\tau}{h}u^{(j+1)}u_t^{(j)}\right]\bigg|_{X_{j+1}} \\
&+ \sum_{j=1}^{\infty} \left[-\frac{1}{2}u_x^{(j)}u_t^{(j)}- \frac{1}{2}u^{(j)}u_{tx}^{(j)}- \frac{1}{2}u_x^{(j-1)}u_t^{(j)}+ \frac{1}{2}u^{(j-1)}u_{tx}^{(j)}- \frac{\tau}{h}u^{(j)}u_t^{(j)}+\frac{\tau}{h}u^{(j-1)}u_t^{(j)}\right]\bigg|_{X_{j}}.
\end{align*}  
\normalsize
Next, we add the above two equations and separate terms for $j=1$ and $j>1$,
\small
\begin{align*}
&\mc{w}_t^TH\mc{w}_{tt} +\sum_{j=1}^{\infty} (u_{tt}^{(j)}, u_{t}^{(j)})_{I_j} \\
=& -\mc{w}_t^T A\mc{w} - \sum_{j=1}^{\infty}  (u_x^{(j)}, u_{tx}^{(j)})_{I_j} \\
&+ \frac{1}{2}\mc{w}_t^T\mb{e}_n\mb{d}_n^T\mc{w} + \frac{1}{2}\mc{w}_t^T\mb{d}_n\mb{e}_n^T\mc{w}+\frac{1}{2}\mc{w}_t^T\mb{e}_n u_{x\Gamma}^{(1)} -\frac{1}{2}\mc{w}_t^T\mb{d}_n u_{\Gamma}^{(1)}   -\frac{\tau}{h}\mc{w}_t^T\mb{e}_n\mb{e}_n^T\mc{w}+\frac{\tau}{h}\mc{w}_t^T\mb{e}_n u_{\Gamma}^{(1)} \\
&+ \left[-\frac{1}{2}u_x^{(1)}u_t^{(1)}- \frac{1}{2}u^{(1)}u_{tx}^{(1)}- \frac{1}{2}u_x^{(0)}u_t^{(1)}+ \frac{1}{2}u^{(0)}u_{tx}^{(1)}- \frac{\tau}{h}u^{(1)}u_t^{(1)}+\frac{\tau}{h}u^{(0)}u_t^{(1)}\right]\bigg|_{X_{1}}\\
&+ \sum_{j=2}^{\infty} \left[-\frac{1}{2}u_x^{(j)}u_t^{(j)}- \frac{1}{2}u^{(j)}u_{tx}^{(j)}- \frac{1}{2}u_x^{(j-1)}u_t^{(j)}+ \frac{1}{2}u^{(j-1)}u_{tx}^{(j)}- \frac{\tau}{h}u^{(j)}u_t^{(j)}+\frac{\tau}{h}u^{(j-1)}u_t^{(j)}\right]\bigg|_{X_{j}}\\
&+ \sum_{j=1}^{\infty} \left[\frac{1}{2}u_x^{(j)}u_t^{(j)}+ \frac{1}{2}u^{(j)}u_{tx}^{(j)}+ \frac{1}{2}u_x^{(j+1)}u_t^{(j)}- \frac{1}{2}u^{(j+1)}u_{tx}^{(j)}- \frac{\tau}{h}u^{(j)}u_t^{(j)}+\frac{\tau}{h}u^{(j+1)}u_t^{(j)}\right]\bigg|_{X_{j+1}}.
\end{align*}  
\normalsize
On the right-hand side, terms on the second and the third line are numerical fluxes at the FD-DG interface, and terms on the fourth and the fifth line are numerical fluxes at the DG inter-element interfaces. After combining terms, we have 
\begin{align*}
\frac{1}{2}\frac{d}{dt}\left[\mc{w}_t^TH\mc{w}_t +\sum_{j=1}^{\infty} (u_{t}^{(j)}, u_{t}^{(j)})_{I_j}\right] 
=&\frac{1}{2}\frac{d}{dt}\left[ -\mc{w}^T A\mc{w} - \sum_{j=1}^{\infty}  (u_x^{(j)}, u_{x}^{(j)})_{I_j} \right.\\
&+(\mb{e}_n^T \mc{w}-u_{\Gamma}^{(1)})(\mb{d}_n^T\mc{w}+u_{x\Gamma}^{(1)})-\frac{\tau}{h}(\mb{e}_n^T\mc{w}-u_{\Gamma}^{(1)})^2\\
&\left.+\sum_{j=2}^{\infty}((u^{(j-1)}-u^{(j)})(u_x^{(j-1)}+u_x^{(j)})-\frac{\tau}{h}(u^{(j-1)}-u^{(j)})^2)|_{X_j}\right].
\end{align*}  
In the final step, we shall prove that the expression in the square bracket on the right-hand side is nonpositive with an appropriate choice of $\tau$. For this, we need additional terms $(\mb{d}_n^T\mc{w})^2$ and $(u_{x\Gamma}^{(1)})^2$, which can be obtained by using the borrowing trick in Lemma \ref{lemB}. We write 
\begin{equation}\label{wAw}
\mc{w}^T A\mc{w} = \mc{w}^T \tilde A\mc{w} + \bp h (\mb{d}_n^T\mc{w})^2.  
\end{equation}
In addition, by using the inverse inequality in Lemma \ref{lemma_inv}, we have 
\begin{equation}\label{trace}
\sum_{j=1}^{\infty}  (u_x^{(j)}, u_{x}^{(j)})_{I_j}  \geq \sum_{j=1}^{\infty}\paratr h ((u_x^{(j)}|_{X_j})^2+(u_x^{(j)}|_{X_{j+1}})^2).
\end{equation}
Combining \eqref{wAw} and \eqref{trace}, we have
\begin{align}
\begin{split}
\mc{w}^T A\mc{w} + \sum_{j=1}^{\infty}  (u_x^{(j)}, u_{x}^{(j)})_{I_j} &\geq  \mc{w}^T \tilde A\mc{w} + \bp h (\mb{d}_n^T\mc{w})^2 + \sum_{j=1}^{\infty}\paratr h ((u_x^{(j)}|_{X_j})^2+(u_x^{(j)}|_{X_{j+1}})^2)\\
& \geq  \mc{w}^T \tilde A\mc{w} +\frac{\tilde\beta h}{2}(\mb{d}_n^T \mc{w}+u_{x\Gamma}^{(1)})^2 + \frac{\paratr h}{2} \sum_{j=2}^{\infty} (u_x^{(j-1)}+u_x^{(j)})^2|_{X_j},
\end{split}
\label{wAwtrace}
\end{align}
where $\tilde\beta=\max(\beta,\paratr)$. Now, for the discrete energy $E_h$, we have  
 \begin{align*}
 E_h =\ & \mc{w}_t^TH\mc{w}_t +\sum_{j=1}^{\infty} (u_{t}^{(j)}, u_{t}^{(j)})_{I_j} + \mc{w}^T A\mc{w} + \sum_{j=1}^{\infty}  (u_x^{(j)}, u_{x}^{(j)})_{I_j} \\
&-(\mb{e}_n^T \mc{w}-u_{\Gamma}^{(1)})(\mb{d}_n^T\mc{w}+u_{x\Gamma}^{(1)})+\frac{\tau}{h}(\mb{e}_n^T\mc{w}-u_{\Gamma}^{(1)})^2\\
&+\sum_{j=2}^{\infty}(-(u^{(j-1)}-u^{(j)})(u_x^{(j-1)}+u_x^{(j)})+\frac{\tau}{h}(u^{(j-1)}-u^{(j)})^2)|_{X_j}\\
\geq &\ \mc{w}_t^TH\mc{w}_t +\sum_{j=1}^{\infty} (u_{t}^{(j)}, u_{t}^{(j)})_{I_j} + \mc{w}^T \tilde A\mc{w} \\
&+\frac{\tilde\beta h}{2}(\mb{d}_n^T \mc{w}+u_{x\Gamma}^{(1)})^2-(\mb{e}_n^T \mc{w}-u_{\Gamma}^{(1)})(\mb{d}_n^T\mc{w}+u_{x\Gamma}^{(1)})+\frac{\tau}{h}(\mb{e}_n^T\mc{w}-u_{\Gamma}^{(1)})^2\\
&+\sum_{j=2}^{\infty}\left(\frac{\tilde\beta h}{2} (u_x^{(j-1)}+u_x^{(j)})^2 - (u^{(j-1)}-u^{(j)})(u_x^{(j-1)}+u_x^{(j)})+\frac{\tau}{h}(u^{(j-1)}-u^{(j)})^2\right)\bigg|_{X_j}.
 \end{align*}
On the right-hand side, the terms on the first line are nonnegative, because $H$ is symmetric positive definite and $\tilde A$ is symmetric positive semidefinite. The terms on the second line and the third line are nonnegative if 
\[
2\sqrt{\frac{\tilde\beta h}{2}\frac{\tau}{h}}\geq 1\Rightarrow \tau\geq\frac{1}{2\tilde\beta}.
\]
As a consequence, we have the discrete energy conservation $\frac{d}{dt}E_h=0$ with $E_h\geq 0$. This completes the proof.
 \end{proof}
 
Equivalently, the semidiscretization \eqref{FD1d}-\eqref{DG1d} can be written in a matrix form
 \begin{equation}\label{FDDG1dM}
 \mc{z}_{tt} = Q\mc{z},
 \end{equation}
where $\mc{z}=[\mc{w}; \mc{u}]$. 
The vectors $\mc{w}$ and $\mc{u}$ consists of the FD solution, and the DG solution on the Lagrange nodes, respectively. The components of $\mc{u}$ can also be interpreted as the coefficients  multiplied with the Lagrange basis functions for the DG solution. The energy conservation can then be expressed as 
 \[
 \frac{d}{dt}E_h=\frac{d}{dt}\left(\mc{z}_t^T \tilde H \mc{z}_{t} - \mc{z}^T \tilde H Q \mc{z}\right)  = 0,
 \]
 where $\tilde H =\begin{bmatrix}
 H & \\
  & M
 \end{bmatrix}$, $H$ is the SBP norm, and $M$ is the DG mass matrix. The matrix $\tilde HQ$ is symmetric negative semidefinite.  For convenience, we define the energy norm $\vertiii{\mc{z}} = \sqrt{E_h}$.

\subsection{Accuracy analysis}
In this section, we derive an a priori error estimate for the semidiscretization \eqref{FD1d}-\eqref{DG1d} by using a combination of the energy method and the normal mode analysis \cite{Gustafsson2013}.  We start with the semidiscretization \eqref{FDDG1dM} and derive the corresponding error equation. Next, we separate the truncation error into three parts, the truncation error in the interior of the FD domain, in the interior of the DG domain, and at the FD-DG interface. The pointwise error due to the truncation error away from the interface is analyzed by the energy method, and the pointwise error due to the truncation error at the interface is analyzed by the normal mode analysis. In the latter, the form of \eqref{FD1d}-\eqref{DG1d} plays an important role. Thus, to make the accuracy analysis precise, we consider a particular case with SBP operators with four order interior stencil and DG local polynomials of degree three.  This choice of matching the accuracy in both discretizations is determined by the fact that when discretizing in space the wave equation in second order form by the SBP operators of accuracy $(2p,p)$, it is often the $p^{th}$ order accurate boundary closure that determines the convergence rate. Because the number of grid points with the boundary closure is independent of the mesh size, the energy estimate predicts a convergence rate $p+1/2$ that is often suboptimal. Sharper error estimates can be derived using the normal mode analysis by analyzing the precise properties of the boundary closure. This approach yields a convergence rate of $p+2$ for many problems, though there are special cases with rates lower or higher than $p+2$, see \cite{Svard2019,Wang2017,Wang2022}.  For the IPDG discretization based on local polynomials of degree $k$, the optimal convergence rate in $l^2$ norm is $k+1$.  This motivates $p+2=k+1$ in the FD-DG method, and our choice corresponds to $p=2$ and $k=3$. In this case, there are four grid points in the SBP FD boundary closure, and four DOFs in each DG element. 

Let $\bs{\ve}=[\ve_1,\ve_2,\cdots]^T$, where $\ve_j(t) = w_j(t) - U(x_j,t)$ is the pointwise error in the finite difference solution at $x_j$. On the DG side, we define $ e^{(j)}= u^{(j)} - U_h^{(j)}$ as the error in element $I_j$, where $U_h^{(j)}=\mathcal{I}U \in\mathcal{P}^k (I_j)$ is the interpolation of $U$ onto the space $\mathcal{P}^k (I_j)$. The interpolation error $U-U_h$ can be estimated by using standard approximation theory and will not be considered in the following analysis.  
When the DG weak form is realized as stencils in the analysis, we use the notation $e^{(j)}_i(t) = e^{(j)}(X_j+(i-1)h/3, t),\ i=1,2,3,4,$ as the pointwise error of the DG solution on the Lagrange nodes, and $\mb e^{(j)}=[e^{(j)}_1,e^{(j)}_2,e^{(j)}_3,e^{(j)}_4]$. By using the Taylor series expansion, we have 
\begin{equation}\label{Taylore}
\mb{e}^{(j)} = {\bar e}^{(j)} [1,1,1,1]^T+\mathcal{O}(h),
\end{equation}
where ${\bar e}^{(j)} = \frac{1}{4}\sum_{i=1}^4 {e}^{(j)}_i$. We also define $\mb{e}=[\mb{e}^{(1)}, \mb{e}^{(2)}, \cdots]^T$.

 We decompose the error as $[\bs{\ve}; \mb{e}]=\bs{\xi}+\bs{\delta}$, 
and write the error equation in two parts, 
\begin{align}
\bs{\xi}_{tt} = Q\bs{\xi}+\mb{T}_{\bs\xi},\label{xieqn} \\
 \normalsize{\bs{\delta}_{tt} = Q\bs{\delta}}+\mb{T}_{\bs\delta},\label{deltaeqn} 
 \end{align}
 where
\[
 \mb{T}_{\xi}= \begin{bmatrix}
 \mb{T}_{FD}\\
 \mb{0}\\
 \mb{T}_{DG}
 \end{bmatrix},\quad \mb{T}_{\delta}= \begin{bmatrix}
 \mb{0}\\
 \mb{T}_{\Gamma}\\
 \mb{0}
 \end{bmatrix}.
\]
 
 The first part $\bs{\xi}$ is due to the truncation error $\mb{T}_{FD}$ and $\mb{T}_{DG}$ resulted from the interior  of the FD and DG discretization, respectively. The vector $\mb{0}$ in the $\bs\xi$-equation has dimension 8-by-1. The second part  $\bs{\delta}$ is due to the truncation error $\mb{T}_{\Gamma}$ at the FD-DG interface, which involves the first four grid points on the FD side and the first element in the DG side. Thus, the length of $\mb{T}_{\Gamma}$ is 8. For convenience, we also introduce the notation $\bs\xi=[\bs\xi^-;\bs\xi^+]$ and $\bs\delta=[\bs\delta^-;\bs\delta^+]$, where the superscripts minus and plus denote the FD part and DG part, respectively. In the following, we derive error estimates for $\bs{\xi}$ and $\bs{\delta}$ separately. 
 
 \subsubsection{Error estimate of $\bs\xi$}
We have the following error estimate by the energy method. 
\begin{theorem}
The error $\bs{\xi}$ in \eqref{xieqn} satisfies 
\[
\vertiii{\bs\xi}\leq C h^4,
\]
where $C$ depends on the sixth derivative of the true solution.
\end{theorem}
\begin{proof}
 On the FD side, the standard fourth order centred stencil $D_+ D_- - \frac{1}{12}(D_+D_-)^2$ is used on grid points $j=5,6,\cdots$.  By using the Taylor series expansion, every component of $\mb{T}_{FD}$ is $\mathcal{O}(h^4)$. In the interior of the DG discretization from element $i$ with $i=2,3,\cdots$, the stencils in $Q$ have a repeated block structure. In matrix form, each block can be written as a 4-by-12 matrix,
\setcounter{MaxMatrixCols}{20}
\[
\small
\frac{1}{h^2}\begin{bmatrix}
 8 &  -36&    72&     296&    -368&  36& -18&      80&     -59&   -18&    9&     -2\\
-34/27& 17/3& -34/3& -943/27& 1576/27& -32&  20& -622/27&  412/27&  16/3& -8/3&  16/27\\
16/27& -8/3&  16/3&  412/27& -622/27&  20& -32& 1576/27& -943/27& -34/3& 17/3& -34/27\\
 -2&    9&   -18&     -59&      80& -18& 36&    -368&     296&    72&  -36&      8
\end{bmatrix}.
\] 
Each row corresponds to a Lagrange node in an element with a 12-point stencil. We compute the truncation error of these stencils by using Taylor series expansion, 
\begin{equation}\label{T_DG_loc}
\mb{T}_{DG}^{loc} = 
\begin{bmatrix}
 \frac{1}{108}U_{xxxx}(X_j,t)\\
  -\frac{1}{324}U_{xxxx}(X_j+h/3,t)\\
   -\frac{1}{324}U_{xxxx}(X_j+2h/3,t)\\
   \frac{1}{108}U_{xxxx}(X_j+h,t) 
   \end{bmatrix}h^2+\begin{bmatrix}
- \frac{2}{27}U_{xxxx}(X_j,t)\\
  \frac{11}{729}U_{xxxx}(X_j+h/3,t)\\
   -\frac{11}{729}U_{xxxx}(X_j+2h/3,t)\\
   \frac{2}{27}U_{xxxx}(X_j+h,t) 
   \end{bmatrix}h^3+\mathcal{O}(h^4).
\end{equation}
Unlike the interior FD stencil, the interior DG stencil is only second order accurate. At first glance, this does not lead to a fourth order convergence rate. However, when multiplying with the local mass matrix, we have 
\footnotesize
\begin{equation}\label{DGtruncationM}
\begin{split}
&M^{loc}\mb{T}_{DG}^{loc}\\
 = &h 
\begin{bmatrix}
8/105 &        33/560&         -3/140 &        19/1680  \\
      33/560   &      27/70      &   -27/560&         -3/140\\   
      -3/140    &    -27/560     &    27/70    &      33/560  \\ 
      19/1680 &       -3/140    &     33/560    &      8/105   
\end{bmatrix}
\left(
\begin{bmatrix}
 \frac{1}{108}U_{xxxx}(X_j,t)\\
  -\frac{1}{324}U_{xxxx}(X_j+h/3,t)\\
   -\frac{1}{324}U_{xxxx}(X_j+2h/3,t)\\
   \frac{1}{108}U_{xxxx}(X_j+h,t) 
   \end{bmatrix}h^2+\begin{bmatrix}
- \frac{2}{27}U_{xxxxx}(X_j,t)\\
  \frac{11}{729}U_{xxxxx}(X_j+h/3,t)\\
   -\frac{11}{729}U_{xxxxx}(X_j+2h/3,t)\\
   \frac{2}{27}U_{xxxxx}(X_j+h,t) 
   \end{bmatrix}h^3+\mathcal{O}(h^4)\right)\\
   =&
   \begin{bmatrix}
 U_{xxxx}(X_j,t)\\
  -U_{xxxx}(X_j+h/3,t)\\
   -U_{xxxx}(X_j+2h/3,t)\\
   U_{xxxx}(X_j+h,t) 
   \end{bmatrix}\frac{h^3}{1440}+\begin{bmatrix}
-\frac{163}{45360} U_{xxxxx}(X_j,t)\\
  \frac{1}{1680}U_{xxxxx}(X_j+h/3,t)\\
   -\frac{1}{1680}U_{xxxxx}(X_j+2h/3,t)\\
 \frac{163}{45360}  U_{xxxxx}(X_j+h,t) 
   \end{bmatrix}h^4+\mathcal{O}(h^5).
\end{split}
\end{equation}
\normalsize
As will be shown, the coefficients in the above expression  lead to a cancellation of error terms, which is key to obtain optimal convergence rate. To see this, we multiply  \eqref{xieqn} by $\bs\xi_t^T \tilde H$ to obtain 
\begin{equation}\label{xisqu_est}
\vertiii{\bs\xi}^2 = 2\bs{\xi}_t \tilde H \mb{T}_{\bs\xi}=2(\bs\xi^-_t)^T H \tilde{\mb{T}}_{FD}+2(\bs\xi^+_t)^T M \tilde{\mb{T}}_{DG},
\end{equation}
 where $\tilde{\mb{T}}_{FD}$ is ${\mb{T}}_{FD}$ appended with 4 zeros at the end, and $\tilde{\mb{T}}_{DG}$ is ${\mb{T}}_{DG}$ appended with 4 zeros in the beginning. By using the Cauchy-Schwarz inequality, we bound the first term in \eqref{xisqu_est} by
 \begin{equation}\label{ve_est}
 2(\bs\xi^-_t)^T H \tilde{\mb{T}}_{FD}\leq 2\|\bs\xi^-_t\|_H\|\tilde{\mb{T}}_{FD}\|_H\leq Ch^4 \|\bs\xi^-_t\|_H.
  \end{equation}
  Because of the second order truncation error \eqref{T_DG_loc}, a direct application of the  Cauchy-Schwarz inequality to the second term in  \eqref{xisqu_est} results in a suboptimal estimate $\sim h^2$. To obtain an optimal estimate $\sim h^4$, we consider $2(\bs\xi^+_t)^T M \tilde{\mb{T}}_{DG}$ within one element $I_j$, that is $(\bs\xi^{+(j)}_t)^T M^{loc} \mb{T}_{DG}^{loc}$. Taking the two leading order terms, we have 
   \begin{equation}\label{e_est_temp}
  \begin{split}
 & | (\bs\xi^{+(j)}_t)^T M^{loc} \mb{T}_{DG}^{loc} | \leq  C h^3|(\bar{\bs\xi}^{+(j)})_t|(|T_1|+h|T_2|). 
  \end{split}
     \end{equation}
 where 
  \small
   \begin{equation*}
  \begin{split}
  T_1 &= U_{xxxx}(X_j,t)-U_{xxxx}(X_j+h/3,t)-U_{xxxx}(X_j+2h/3,t)+U_{xxxx}(X_j+h,t),\\
  T_2&=-\frac{163}{45360}U_{xxxxx}(X_j,t)+\frac{1}{1680}U_{xxxxx}(X_j+h/3,t)-\frac{1}{1680}U_{xxxxx}(X_j+2h/3,t)+\frac{163}{45360}U_{xxxxx}(X_j+h,t).
    \end{split}
     \end{equation*}
  \normalsize  
    In the above, by an analogue of \eqref{Taylore}, we have used  
    \[
    \bs\xi^{+(j)}_t=[(\bs\xi^{+(j)}_1)_t,(\bs\xi^{+(j)}_2)_t,(\bs\xi^{+(j)}_3)_t,(\bs\xi^{+(j)}_4)_t]^T=(\bar{\bs\xi}^{+(j)})_t [1,1,1,1]^T+\mathcal{O}(h), 
\]
where $(\bar{\bs\xi}^{+(j)})_t = \frac{1}{4}\sum_{i=1}^4 (\bs\xi^{+(j)}_i)_t$. Substituting the Taylor series expansion 
 \begin{align*}
 U_{xxxx}(X_j+h/3,t) &=  U_{xxxx}(X_j,t)+\frac{h}{3}U_{xxxxx}(X_j,t)+\frac{h^2}{9}U_{xxxxxx}(X_j,t)+\mathcal{O}(h^3),\\
 U_{xxxx}(X_j+2h/3,t) &=  U_{xxxx}(X_j,t)+\frac{2h}{3}U_{xxxxx}(X_j,t)+\frac{4h^2}{9}U_{xxxxxx}(X_j,t)+\mathcal{O}(h^3),\\
 U_{xxxx}(X_j+h,t) &=  U_{xxxx}(X_j,t)+hU_{xxxxx}(X_j,t)+h^2U_{xxxxxx}(X_j,t)+\mathcal{O}(h^3),
 \end{align*}
to $T_1$, we find that the first two terms in the expansions cancel, and  obtain $|T_1|\leq C_1 h^2|U_{xxxxxx}(X_j,t)|$. Similarly, since the sum of the coefficients in $T_2$ is zero, the first term in the Taylor expansions of $U_{xxxxx}$ in $T_2$ cancels, and leads to $|T_2|\leq C_2 h|U_{xxxxxx}(X_j,t)|$. Consequently, the estimate \eqref{e_est_temp} becomes
  \begin{equation*}
  | (\bs\xi^{+(j)}_t)^T  M^{loc} \mb{T}_{DG}^{loc} | \leq C h^5|(\bar{\bs\xi}^{+(j)})_t| |U_{xxxxxx}(X_j,t)|\leq C h^5\sqrt{\sum_{i=1}^4 ((\bs\xi^{+(j)}_i)_t)^2} |U_{xxxxxx}(X_j,t)|. 
  \end{equation*}
 Summing the contribution from all DG elements leads to 
 \begin{equation}\label{e_est}
 2(\bs\xi^{+}_t)^T  M \tilde{\mb{T}}_{DG}\leq Ch^4 \|\bs\xi^{+}_t \|,
 \end{equation}
where we have included the dependence of $U_{xxxxxx}$ into $C$. Finally, we combine \eqref{xisqu_est}, \eqref{ve_est} and \eqref{e_est} to obtain 
 \[
 \vertiii{\bs\xi}^2\leq Ch^4(\|\bs\xi^{-}_t \|_H+\|\bs\xi^{+}_t\|)\leq Ch^4\vertiii{\bs\xi}.
 \]
 Dividing $\vertiii{\bs\xi}$ on both sides completes the proof. 
\end{proof}
A common approach of deriving error estimates for DG discretizations are based on the weak form using the Galerkin orthogonality and special projection operators at the inter-element interfaces, see \cite{Grote2006,Hesthaven2008}. In the above, we have taken a different approach by estimating the errors in the coefficients of the DG basis functions and obtained expected convergence results. In this way, the accuracy analysis is performed in the same framework for both the FD and DG discretizations. 

\subsubsection{Error estimate of $\bs\delta$}
At the FD-DG interface, the first four grid points on the FD side and the first element on the DG side are affected by the interface closure. As a consequence, the interface stencils can be written as the following 8-by-14 matrix, 
\begin{equation*}
\footnotesize{
\frac{1}{h^2}
\begingroup
\setlength\arraycolsep{1.5pt}
\begin{bmatrix}
-4/49  &        64/49  &      -118/49    &      59/49    &       0       &      -9/49     &      8/49     &      0     &      0    &          0    &          0    &          0     &         0 &             0 \\      
       0  &           -4/43   &       59/43   &     -110/43   &       59/43   &       32/43   &      -36/43    &       0     &         0      &        0       &       0       &       0     &         0      &        0\\       
       0    &          0     &         0      &        1     &        -2       &     -13/59     &     72/59     &      0     &         0     &         0      &        0      &        0    &          0      &        0   &\\    
       0     &         0&             -9/17      &    32/17    &     -13/17    &   -1166/17   &     1024/17     &    216/17    &    -108/17    &      24/17    &       0    &          0     &         0    &          0\\       
       0     &         0     &         8/3     &     -12      &       24       &     976/3     &    -368      &       36      &      -18      &       80     &       -59      &      -18       &       9      &       -2    \\   
       0     &         0     &       -34/81     &     17/9     &     -34/9   &     -3203/81     &   1576/27    &     -32     &        20   &        -622/27   &      412/27    &      16/3   &        -8/3  &         16/27  \\  
       0     &         0     &        16/81      &    -8/9      &     16/9     &    1412/81    &    -622/27    &      20     &       -32      &     1576/27     &   -943/27      &   -34/3    &       17/3   &       -34/27  \\  
       0     &         0    &         -2/3    &        3     &        -6     &      -199/3     &      80      &      -18       &      36     &      -368      &      296     &        72    &        -36     &         8       
\end{bmatrix}.
\endgroup
}
\end{equation*}
The eight rows correspond to grid points $x_4, x_3, x_2, x_1$ on the FD side, and $X_1, X_1+h/3, X_1+2h/3, X_1+h$ on the DG side. By using the Taylor series expansion, we compute the truncation errors on these eight  points and obtain 
\begin{equation}\label{TGamma}
\mb{T}_{\Gamma} = \begin{bmatrix}
-11/588,
      -5/516, 
       1/12,
    -337/612,
     209/108,
    -893/2916,
     407/2916,
     -17/36    
\end{bmatrix}h^2 U_{xxxx}|_{\Gamma}.
\end{equation}
Here, the truncation error is only second order, but it is important to note that the length of $\mb{T}_{\Gamma}$ is always eight independent of $h$. A straightforward application of the energy method to the error equation leads to a convergence rate 2.5 in the energy norm. In the following, we derive a sharp estimate for $\bs{\delta}$ by the normal mode analysis \cite{Gustafsson2013}, which has also been used for deriving error estimates for the FD discretization \cite{Wang2017}.

\begin{theorem}
The error \eqref{deltaeqn} satisfies the error estimate 
\begin{equation}\label{delta_EST}
\sqrt{\int_{0}^T \|\bs\delta\|^2_h dt}\leq Ch^4,
\end{equation}
where $C$ depends on the final time $T$ and the fourth derivative of the true solution at the interface. 
\end{theorem}

\begin{proof}
On the FD side, the error equation in the interior $j=5,6,\cdots$ takes the form 
\[
(\delta^-_j)_{tt} = (D_+ D_- - \frac{1}{12}(D_+D_-)^2)\delta^-_j,
\]
where $D_+$ and $D_-$ are the standard forward and backward difference operators, respectively. Next, we perform a Laplace transform in time $t$, and obtain
\[
s^2\hat \delta^-_j = (D_+ D_- - \frac{1}{12}(D_+D_-)^2)\hat \delta^-_j,
\]
 where $s$ is the time dual, and the hat-variables are in Laplace space. The general solution to the above equation is 
\begin{equation}\label{gelFD}
\hat \delta^-_j = \sigma_1 \kappa_1^{j-3}+\sigma_2 \kappa_2^{j-3},\ j=3,4,5,\cdots,
\end{equation}
where $\kappa_1 = 1-\tilde s$ and $\kappa_2= 7-4\sqrt{3}+\mathcal{O}(\tilde s^2)$ are the two admissible solutions to the corresponding characteristic equation. The two unknown coefficients $\sigma_1$, $\sigma_2$, and the two pointwise errors $\hat\delta^-_1$, $\hat\delta^-_2$ will be determined by the numerical scheme at the interface.

The DG discretization \eqref{DG1d} can be written in a matrix form 
\begin{equation}\label{M2}
\mathbf{u}^{(j)}_{tt} = D_1\mathbf{u}^{(j-1)}+ D_2\mathbf{u}^{(j)} +D_3\mathbf{u}^{(j+1)}\approx U_{xx}(\mathbf{X}^{(j)},t).
\end{equation}
Here, $\mathbf{u}^{(j)}=[u^{(j)}_1,u^{(j)}_2,u^{(j)}_{3},u^{(j)}_{4}]^T$ consists of the unknown coefficients for the DG solution in $I_j$. The matrices $D_i=(M^{loc})^{-1}A_i,\ i=1,2,3$, where the $4$-by-$4$ matrix $M^{loc}$ is the local mass matrix, and is symmetric positive definite. The vector $\mathbf{X}^{(j)}=[X_j, X_j+h/3, X_j+2h/3, X_{j+1}]^T$ consists of the Lagrange nodes in $I_j$.

Here, the matrices $D_1,D_2,D_3$ are realized as difference stencils. The interior error equation on the DG side is 
\begin{equation}\label{DGerrInt}
\bs{\delta}^{+(j)}_{tt} = D_1\bs{\delta}^{+(j-1)}+ D_2\bs{\delta}^{+(j)} +D_3\bs{\delta}^{+(j+1)},\quad j=2,3,\cdots. 
\end{equation}
The Laplace transform of \eqref{DGerrInt} in time is 
\begin{equation}\label{DGerrIntL}
\tilde s^2 \boldsymbol{\hat \delta}^{+(j)} = \tilde D_1\boldsymbol{\hat  \delta}^{+(j-1)}+ \tilde D_2\boldsymbol{\hat  \delta}^{+(j)} +\tilde D_3\boldsymbol{\hat  \delta}^{+(j+1)},\quad j=2,3,\cdots,
\end{equation}
where $\tilde s=sh$, and $\tilde D_i = h^2 D_i$, $i=1,2,3$. We substitute the ansatz $\boldsymbol{\hat \delta}^{+(j)}=\alpha^{j-1}\mathbf{z},\ j=1,2,\cdots,$ in \eqref{DGerrIntL}, and obtain 
\begin{equation}\label{ls0}
\tilde s^2 \alpha^{j-1}\mathbf{z} = \tilde D_1 \alpha^{j-2}\mathbf{z}+\tilde D_2\alpha^{j-1}\mathbf{z}+\tilde D_3 \alpha^{j}\mathbf{z},\quad j=2,3,\cdots.  
\end{equation}
If $\alpha\neq 0$, we have 
\begin{equation}\label{ls}
(\tilde D_1 + (\tilde D_2-\tilde s^2 I)\alpha+\tilde D_3 \alpha^2)\mathbf{z} =0.
\end{equation}
A solution exists if $\det(\tilde D_1 + (\tilde D_2-\tilde s^2 I)\alpha+\tilde D_3 \alpha^2)=0$. By a direct calculation, we find that the determinant is a sixth order polynomial in $\alpha$. At $\tilde s=0$, the six roots are 0,0,0.1390,1,1,7.1943. A perturbation analysis with $\tilde s>0$ shows that there are two nonzero admissible roots $\alpha_1=0.1390-4.3780\times 10^{-4}\tilde s^2$, $\alpha_2 = 1-\tilde s$.  For each admissible root, we compute the corresponding eigenvector by \eqref{ls},
\begin{align*}
 \mathbf{z}_1 = \begin{bmatrix}
-7.1943\\
   2.7016\\
  -0.6368\\
   1.0000
   \end{bmatrix}-\begin{bmatrix}
   0.0227\\
   0.0821\\
   0.0480\\
   0
   \end{bmatrix}\tilde s^2,\quad \mathbf{z}_2 = \begin{bmatrix}
1\\
  1\\
  1\\
   1
   \end{bmatrix}+\begin{bmatrix}
   1\\
   2/3\\
   1/3\\
   0
   \end{bmatrix}\tilde s.
\end{align*}
Next, we consider the case when $\alpha=0$. The relation \eqref{ls0} is reduced to $\tilde D_1 \mathbf{z}=0$ for $j=2$, which has two solutions 
\begin{align*}
 \mathbf{z}_3 = \begin{bmatrix}
9/2\\
   1\\
  0\\
   0
   \end{bmatrix},\quad\mathbf{z}_4 = \begin{bmatrix}
   -9\\
   0\\
   1\\
   0
   \end{bmatrix},
   \end{align*}
   corresponding to $\alpha_3=\alpha_4=0$. The general solution to \eqref{DGerrIntL} can be written as 
\begin{equation}\label{gelDG}
\boldsymbol{\hat \delta}^{+(j)}=c_1\alpha^{j-1}_1\mathbf{z}_1+c_2\alpha^{j-1}_2\mathbf{z}_2+c_3\alpha^{j-1}_3\mathbf{z}_3+c_4\alpha^{j-1}_4\mathbf{z}_4j,\quad j = 1,2,\cdots,
\end{equation}
where $c_i,\ i=1,2,3,4$ are the unknown coefficients. Note that $\mathbf{z}_3$ and $\mathbf{z}_4$ only have contribution when $j=1$.

We use the general solutions \eqref{gelFD} and \eqref{gelDG} in the interface stencils to obtain the set of eight error equations for the FD-DG interface, 
\begin{equation}\label{BS}
C(\tilde s) \mathbf{Z}=h^2 \hat{\mathbf{T}}_{\Gamma},
\end{equation}
where $\mathbf{Z} =[\sigma_1,\sigma_2,\hat\delta^-_1,\hat \delta^-_2,c_1,c_2,c_3,c_4]^T$ and $\hat{\mathbf{T}}_{\Gamma}\sim h^2$ is the Laplace transform of ${\mathbf{T}}_{\Gamma}$  in \eqref{TGamma}. If $C(\tilde s)$ is invertible for all $Re(\tilde s)\geq 0$, then the determinant condition is satisfied \cite{Gustafsson2013} and $|\mathbf{Z}|\sim h^4$ follows. Otherwise, the behaviour of $C(\tilde s)$ in the vicinity of $\tilde s= 0$ shall be analyzed.

We substitute $\tilde s=0$ into the matrix $C(\tilde s)$ and find that $C(0)$ is singular with one eigenvalue equal to zero. To this end, we consider $C(\tilde s)=C(0)+\tilde sC'(0)+\mathcal{O}(\tilde s^2)$. Let $\mathcal{U}\Sigma\mathcal{V}^T$ be the singular value decomposition of $C(0)$. By direct computation, we have $\mathcal{U}^T \hat{\mathbf{T}}_{\Gamma} =[-2.0449,-0.1253,-0.3378,-0.2173$,\\
$0.1352,-0.1218,-0.0082,0]^T$. It is important to note that the last component is equal to zero, which means that $ \hat{\mathbf{T}}_{\Gamma}$ is in the column space of $C(0)$. Consequently, by Lemma 3.4 from \cite{Nissen2012}, the solution $ \mathbf{Z} $ can be bounded as 
\begin{equation}\label{Z_est}
|\mathbf{Z}|\leq Ch^2|\hat{\mathbf{T}}_{\Gamma}|\leq Ch^4,
\end{equation}
 for some constant $C$.  

In the last step, we sum all contributions from \eqref{gelFD} and \eqref{gelDG}. For the error in the FD discretization, we have
\begin{align}\label{delta_m_temp}
\begin{split}
\|\hat{\bs\delta}^-\|_h^2 &= h(|\hat\delta^-_1|^2+|\hat\delta^-_2|^2)+ h\sum_{j=3}^{\infty}|\hat\delta^-_j|^2\\
&=h(|\hat\delta^-_1|^2+|\hat\delta^-_2|^2)+ h\sum_{j=3}^{\infty}|\sigma_1\kappa_1^{j-3}+\sigma_2\kappa_2^{j-3}|^2\\
&\leq h(|\hat\delta^-_1|^2+|\hat\delta^-_2|^2)+ h|\sigma_1|^2\frac{1}{1-|\kappa_1|^2}+h|\sigma_2|^2\frac{1}{1-|\kappa_2|^2}.
\end{split}
\end{align}
For the three terms on the right-hand side, the first and third term  can easily be bounded by using \eqref{Z_est}. More precisely, because of $|\hat\delta_1^-|,|\hat\delta_2^-|,|\sigma_2|\leq Ch^4$, we have  
\[
h(|\hat\delta^-_1|^2+|\hat\delta^-_2|^2)\leq Ch^9, \quad h|\sigma_2|^2\frac{1}{1-|\kappa_2|^2}\leq  Ch^9.
\]
The second term in \eqref{delta_m_temp} contains the slowly-decaying component $\kappa_1=1-\tilde s$. To bound this term, we use Lemma 2 in \cite{Wang2017}, which states
\begin{equation}\label{kappa1_est}
\frac{1}{1-|\kappa_1|^2}\leq \frac{1}{2\eta h},
\end{equation}
where $\eta = Re(s)>0$ is a constant independent of $h$. Consequently, the second term in \eqref{delta_m_temp} is bounded as 
\[
 h|\sigma_1|^2\frac{1}{1-|\kappa_1|^2} \leq C h^8,
 \]
and we have 
\begin{equation}\label{delta_m}
\|\hat{\bs\delta}^-\|_h^2\leq Ch^8. 
\end{equation}

Next, we consider the error in the DG discretization. We have
\begin{align}\label{delta_p_temp}
\begin{split}
\|\hat{\bs\delta}^+\|_h^2 &= \frac{h}{4}\sum_{j=1}^{\infty} \left|\sum_{i=1}^4 c_i\alpha_i^{j-1}\mb{z}_i\right|^2\leq Ch \sum_{i=1}^4 |c_i|^2 |\mb{z}_i|^2\left|\frac{1}{1-|\alpha_i|^2}\right|.
\end{split}
\end{align}
To bound the right-hand side, we use $|c_i|\leq Ch^4,\ i=1,2,3,4$ from \eqref{Z_est}. For the terms with $\alpha_i$, they can be bounded independent of $h$ for $i=1,3,4$. For $\alpha_2=1-\tilde s$, we use again Lemma 2 from \cite{Wang2017} to obtain $\frac{1}{1-|\alpha_2|^2}\leq \frac{1}{2\eta h}$. Since $|\mb{z}_i|,\ i=1,2,3,4$ are independent of $h$, we have 
\begin{align}\label{delta_p}
\|\hat{\bs\delta}^+\|_h^2 \leq Ch^8.
\end{align}
Combining the two estimates \eqref{delta_m} and \eqref{delta_p}, we obtain the estimate for $\bs\delta$ in Laplace space, $\|\hat{\bs\delta}\|_h^2 \leq Ch^8$. By using Parseval's relation and the argument   \textit{future cannot affect past} \cite{Gustafsson2013}, we obtain the final estimate \eqref{delta_EST} in physical space.  
\end{proof}

\section{Numerical treatment at the FD-DG interface in 2D}\label{sec_FD_DG_2D}
In this section, we present an FD-DG discretization for the wave equation in two space dimension. Here, the FD-DG interface is a line segment, where the FD solution is pointwise and the DG solution is a piecewise polynomial. In addition, the DOFs from the FD and DG sides may not coincide. As a consequence, interpolation or projection is needed for coupling the FD and DG solutions, which shall not destroy the stability and accuracy property. 

Our model problem is the wave equation 
\begin{align}\label{wave2d}
U_{tt} &= \nabla \cdot b\nabla U,\quad (x,y)\in\Omega,\ t\in (0,T],
\end{align}
with suitable initial and boundary conditions. The spatial domain consists of two subdomains  $\Omega=\Omega_1\cup\Omega_2$, where $\Omega_1=[0,1]\times [0,1]$ and $\Omega_2=[0,1]\times [-1,0]$ with an   interface $\Gamma=\Omega_1\cap\Omega_2$. The material parameter $b$ is piecewise constant,
\begin{equation*}
b = \begin{cases}
b_1, \quad \text{in}\ \Omega_1, \\
b_2, \quad \text{in}\ \Omega_2, 
\end{cases}
\end{equation*}
where $b_1\neq b_2$ are positive constants. We consider interface conditions that prescribe continuity of pressure and continuity of normal flux,
\begin{align}
U(x,0^-,t)&=U(x,0^+,t),\label{inter1}\\
b_1U_y(x,0^-,t)&=b_2U_y(x,0^+,t).\label{inter2}
\end{align}
We assume that the initial and boundary data are sufficiently smooth and compatible in each subdomain.

We discretize \eqref{wave2d} in space by the SBP operators in $\Omega_1$ and the IPDG in $\Omega_2$, and impose the interface conditions \eqref{inter1}-\eqref{inter2} weakly. On the interface $\Gamma$, the FD solution is pointwise and the DG solution is a piecewise polynomial. For the numerical fluxes on the FD side, we need to interpret the DG solution pointwise; while on the DG side, we need the FD solution in the form of piecewise polynomial. This poses a significant challenge in designing numerical fluxes. To overcome this challenge, we construct projection operators for the FD and DG solutions on the interface. 
In Section \ref{sec_po}, we present the required properties of the projection operators for energy stability and outline the main procedure of constructing these operators. After that, we derive numerical fluxes and prove energy stability in Section \ref{sec_es}.  Finally, in Section \ref{sec_ta} we analyze the truncation error of the numerical interface scheme and identify demands for obtaining optimal convergence. By optimal convergence, we mean that when matching the order of accuracy of the FD and DG discretization, the overall convergence rate is the same as when one method is used in the entire domain. As an example, using the SBP operator with order of accuracy (4,2) in $\Omega_1$ and the IPDG with local polynomials of degree three, the optimal convergence rate for the overall semidiscretization is fourth order. 

\subsection{Projection operators}\label{sec_po}
The projection operators used in this work are inspired by the norm compatible projection operators of Kozdon and Wilcox \cite{Kozdon2016}. For energy stability, we impose the same type of constraints on the projection operators as in \cite{Kozdon2016}. However, the accuracy constraints are different. In \cite{Kozdon2016}, with the model problem of the wave equation in the first order form, it is enough to require the projection operators to mimic the accuracy of the SBP FD stencils. For the wave equation in the second order form, however, the same approach leads to a suboptimal convergence rate \cite{Wang2016} at an FD-FD interface. The optimal convergence rate is recovered by using two pairs of projection operators with improved accuracy property and carefully designed numerical fluxes \cite{Almquist2019}. Similarly, we impose this type of accuracy constraints on the projection operators. 

\begin{figure}
\centering
\includegraphics[trim={0cm 0cm 0cm 0cm},clip,width=0.7\textwidth]{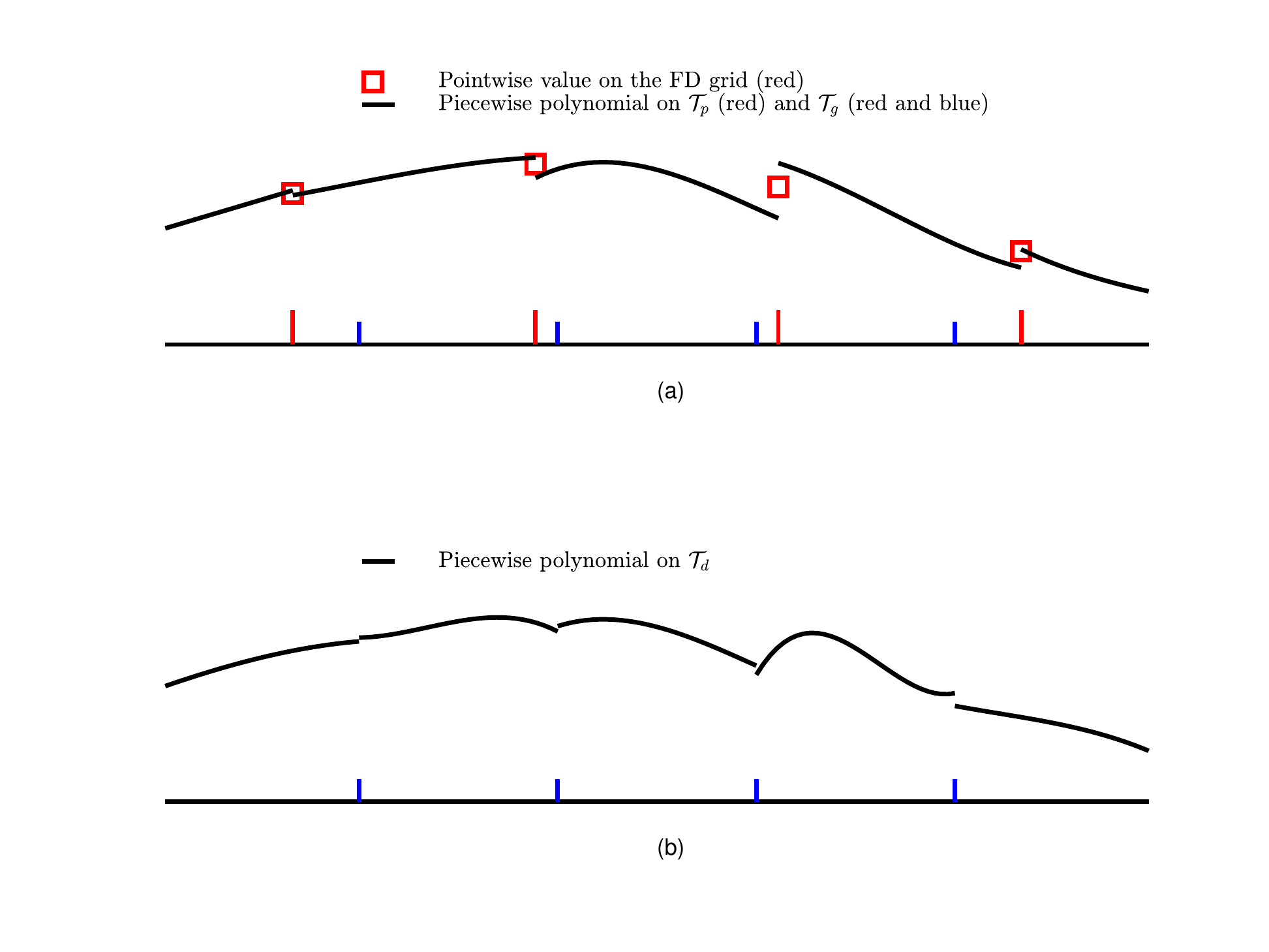}
\caption{ (a) The FD pointwise values on the FD grid, and the corresponding piecewise polynomial on mesh $\mathcal{T}_f$ whose element boundaries are defined by the grid points on the FD grid. The piecewise polynomial is the same on mesh $\mathcal{T}_g$ whose element boundaries are the union of those on $\mathcal{T}_f$ and $\mathcal{T}_d$. (b) The piecewise polynomial on the DG mesh $\mathcal{T}_d$.}
\label{GlueGrid}
\end{figure}

We illustrate in Figure \ref{GlueGrid} how we use projection operators to connect the FD and DG discretizations on the interface. Notation-wise, we use $\mb{x}_f$ to denote a uniform FD grid with $n$ grid points, and $\mathcal{T}_p$ to denote the mesh with element boundaries defined by the grid points on $\mb{x}_f$. In this case, pointwise values are defined on $\mb{x}_f$ and piecewise polynomials are defined on $\mathcal{T}_{p}$. For a piecewise polynomial of degree $q$, there are  $q+1$ DOFs in each element, including two DOFs on the element boundaries and $q-1$ DOFs in the interior of the element. Similarly, we use the notations $\mathcal{T}_{d}$ for the mesh associated with the DG discretization, and $\mathcal{T}_{g}$ for the mesh with element boundaries defined as the union of all element boundaries on $\mathcal{T}_{p}$ and $\mathcal{T}_{d}$.

 As an example, consider the smooth function $f=0.3\cos(\pi x)+0.7\sin(\pi x+2)$ defined on the interface $[-1,1]$. Let vector $\mb{f}$ contain the pointwise evaluation of $f$ on the uniform FD grid $\mb{x}_f$. For clear visualization, we plot only part of $\mb{f}$ in Figure \ref{GlueGrid}a with red-coloured squares.  We use a projection operator $P_{f2p}$ to transform the pointwise values $\mb{f}$ to a piecewise polynomial $\mathcal{P}\mb{f}$ of degree $q=3$, that is,  $\mathcal{P}\mb{f}$ is a cubic polynomial on all the elements on $\mathcal{T}_{p}$ with coefficient vector $\mb{f}_p= P_{f2p} \mb{f}$, see also Figure \ref{GlueGrid}a. The coefficient vector takes the form $\mb{f}_p=[\mb{f}_p^{(1)},\mb{f}_p^{(2)},\cdots,\mb{f}_p^{(n-1)}]^T$, where $\mb{f}_p^{(i)}$ contains the polynomial expansion coefficients in element $i$, i.e. $[x_{i-1},x_i]$.  Here, $\mathcal{P}\mb{f}$ in element $i$ can be expressed as $\sum_{j=0}^{q}  (\mb{f}_p^{(i)})_j \phi^{(i)}_j$, where $\phi^{(i)}_j$ are the local Lagrange basis functions. There is no requirement of continuity at the interface between two adjacent elements. Similarly, the operator $P_{p2f}$ transforms a piecewise polynomial of degree $q$ back to pointwise values on $\mb{x}_f$.  We define the projection errors as
\begin{equation}\label{projerr}
 \mb{e}_{f2p} = \mb{f}_p - \tilde{\mb{f}}_p,\quad \mb{e}_{p2f}=P_{p2f} \tilde{\mb{f}}_p - \mb{f}, 
 \end{equation}
 where $\tilde{\mb{f}}_p$ contains the pointwise evaluation of $f$ on the Lagrange nodes associated with the FD grid. 
 
In general, the element boundaries on $\mathcal{T}_{p}$ do not coincide with those on $\mathcal{T}_{d}$. To this end, we ultilize a glue mesh $\mathcal{T}_{g}$ whose element boundaries are a union of the element boundaries  on $\mathcal{T}_{p}$ and $\mathcal{T}_{d}$, see Figure \ref{GlueGrid}a. As an intermediate step, we perform a basis transformation for the piecewise polynomial on $\mathcal{T}_{p}$ to $\mathcal{T}_{g}$, and use the projection operator $P_{p2g}$ to obtain the weights of the piecewise polynomial. Since the function space $V_h$ defined in  \eqref{Vh} on $\mathcal{T}_{p}$ is a subset of that on $\mathcal{T}_{g}$, the piecewise polynomial itself remains unchanged, illustrated by the fact that there is only one piecewise polynomial plotted in Figure \ref{GlueGrid}a. Next, we perform another basis transformation from $\mathcal{T}_{g}$ to  $\mathcal{T}_{d}$ by the  projection operator $P_{g2d}$ to obtain the piecewise polynomial on $\mathcal{T}_{d}$, see Figure \ref{GlueGrid}b. Analogously, a  piecewise polynomial on $\mathcal{T}_{d}$  can be transformed back to pointwise values on the FD grid by using the operators $P_{d2g}, P_{g2p}$ and $P_{p2f}$. For stability and accuracy, we require that the above projection operators to satisfy a set of constraints, which are presented below from the stability and accuracy perspective.

\paragraph{Stability requirement}
The SBP norm $H$ is associated with the FD grid. Similarly, the mass matrices $M_p, M_g, M_d$ are norms defined on $\mathcal{T}_p$, $\mathcal{T}_g$, $\mathcal{T}_d$, respectively. We note that all four matrices are symmetric positive definite. In particular, $H$ is diagonal, and $M_p, M_g, M_d$ are block-diagonal. The norm compatibility defined below is essential for proving energy stability of the overall semidiscretization.

\begin{definition}[Norm compatibility]
The projection operators are said to be norm compatible if they satisfy 
\begin{align*}
H P_{p2f} &= (M_p P_{f2p})^T,\\
 M_p P_{g2p}& = (M_g P_{p2g})^T,\\
 M_g P_{d2g}& = (M_d P_{g2d})^T.
\end{align*}
\end{definition}

An immediately consequence of the norm compatibility property is stated in the following corollary. 
\begin{corollary}
Let $P_{f2d}=P_{g2d} P_{p2g} P_{f2p}$ and $P_{d2f}=P_{p2f} P_{g2p} P_{d2g}$. We have 
\begin{align}\label{nc_fddg}
H P_{d2f} &= (M_d P_{f2d})^T.
\end{align}
\end{corollary}
The operator $P_{f2d}$ transforms the pointwise values on the FD grid directly to the polynomial expansion coefficients on the DG side. Similarly, the other operator $P_{d2f}$ transforms a discontinuous piecewise polynomial on the DG side to pointwise values on the FD grid. They satisfy the norm compatibility with respect to the SBP norm and the DG mass matrix. In Section \ref{sec_es}, we use the operators $P_{f2d}$ and $P_{d2f}$ in the semidiscretization, and the relation \eqref{nc_fddg} in the stability analysis. We remark that a similar norm compatibility property is also required for the stability at an FD-FD nonconforming interface, see \cite{Almquist2019,Mattsson2010,Wang2018,Wang2016}. 

\paragraph{Accuracy requirement}
First, we consider the projection operators $P_{f2p}$ and $P_{p2f}$. It is natural to require that the errors \eqref{projerr} vanish for polynomials up to a certain degree, and this requirement is different for the grid points in the interior of the interface and the grid points near the edges of the interface. In the interior, the projection operators are based on centred stencils with even order of accuracy $q_i$. Equivalently, the projection error \eqref{projerr} in the interior is zero for polynomials of degree up to $q_i-1$. However, centred stencils cannot be used near the edges because of a lack of grid points on one side of the stencils. Instead, one-sided stencils are used as the closure for a few grid points near the edges. Because of the norm compatibility requirement, the order of accuracy $q_e$ near the edges is often lower than $q_i$. We denote the order of accuracy of the projection operators as $(q_i, q_e)$.

Next, we consider the other four projection operators, $P_{p2g}$, $P_{g2p}$, $P_{g2d}$, $P_{d2g}$. These are  operators for basis transformation between function spaces, and can be constructed in a straightforward way. Since the piecewise polynomial degree does not change, the projections do not change the order of accuracy. In other words, the orders of accuracy of $P_{d2f}$ and $P_{f2d}$ are determined by $P_{f2p}$ and $P_{p2f}$. 

In \cite{Kozdon2016}, projection operators with order of accuracy $(p_i=2p,p_e=p)$ were constructed and were used with SBP FD operators with the same order of accuracy $(2p,p)$. In numerical experiments, almost $p+1$ convergence rate was observed for the wave equation in the first order form, which is considered as optimal. When the same operators were used to solve the wave equation in the second order form, $p+1$ convergence rate was observed \cite{Wang2018}, which is one order lower than $p+2$ that is often seen for problems with FD-FD conforming interfaces \cite{Duru2014V,Mattsson2008}. A straightforward strategy is to use projection operators with an improved order of accuracy $(p_i=2p,p_e=p+1)$, but in \cite{Lundquist2018} it was proved that there exists no such operator with the norm compatibility property.  In \cite{Almquist2019}, it was found that by using two pairs of interpolation operators, the optimal $p+2$ convergence rate was obtained for the wave equation in the second order form with FD-FD  nonconforming interfaces. Inspired by this work, we construct two pairs of projection operators for the FD-DG interface. 

More precisely, we have the first pair $P_{d2f}^g$ and $P_{f2d}^b$, and the second pair $P_{d2f}^b$ and $P_{f2d}^g$. The superscripts $b$ and $g$, denoting \textit{bad} and \textit{good}, indicate order of accuracy $(2p,p)$ and $(2p,p+1)$, respectively. The two pairs of projection operators are independent from each other, and both pairs satisfy the norm compatibility property \eqref{nc_fddg}. To construct each pair, we set unknowns in only one operator and determine the other one by the norm compatibility condition. The unknowns are then computed by using the accuracy requirement for both operators. 

In Section \ref{sec_es}, we use these operators in the semidiscretization and prove stability. After that, we analyze the truncation error of the semidiscretization in Section \ref{sec_ta}.

\subsection{Numerical fluxes and energy stability}\label{sec_es}

We discretize $\Omega_1$ by a Cartesian grid with $n$ grid points in each spatial direction. Let \[
\mb{w}=[w_{11},w_{12},\cdots,w_{1n}, w_{21}, w_{21}, \cdots, w_{2n},\cdots, w_{n1},w_{n2},\cdots,w_{nn}]^T
\]
 be the finite difference solution with $w_{ij}\approx U(x_i, y_j)$.  The semidiscretization reads
\begin{equation}\label{semiFD}
\begin{split}
\mb{w}_{tt} = b_1\mathcal{D}\mb{w} &+\mathcal{H}^{-1}\left[\f{1}{2}b_1\mb{d}_{\Gamma} H (\mb{w}_{\Gamma}-P_{u2w}^{g}\mb{u}_{\Gamma})-\f{1}{2} \mb{e}_{\Gamma} H (b_1\mb{d}_{\Gamma}^{T}\mb{w}+P_{u2w}^{b}b_2\mb{u}_{\Gamma y})\right] \\
&- \mathcal{H}^{-1}\left[\f{b_1\tauw}{h_1} \mb{e}_{\Gamma}H (\mb{w}_{\Gamma}-P_{u2w}^{g}\mb{u}_{\Gamma})+\f{b_2\sigmaw}{h_2} \mb{e}_{\Gamma}H P_{u2w}^{b}(P_{w2u}^{g}\mb{w}_{\Gamma}-\mb{u}_{\Gamma})\right].
\end{split}
\end{equation}
In the first term on the right-hand side of \eqref{semiFD}, the operator $\mathcal{D}=[D\otimes I+I\otimes D]$ approximates the Laplacian, where $D$ is a second derivative SBP operator defined in \eqref{D2} and $I$ is the identity operator. The remaining terms are SAT for the interface conditions \eqref{inter1}-\eqref{inter2}. More precisely, the first SAT imposes continuity of solution  \eqref{inter1}, where  $\mb{w}_{\Gamma}$ is the FD solution on the interface $\Gamma$, and $\mb{u}_{\Gamma}$ is the DG solution evaluated on the Lagrange nodes. The operator $P_{u2w}^{g}$ projects the DG solution to the FD grid on $\Gamma$. The quantity $\frac{1}{2}\mathcal{H}^{-1}b_1\mb{d}_{\Gamma}$ with  
$\mb{d}_{\Gamma}=-I\otimes \mb{d}_1$ and $\mathcal{H}=H\otimes H$ is the weight for the penalization of $\mb{w}_{\Gamma}-P_{u2w}^{g}\mb{u}_{\Gamma}$. 
Similarly, the second SAT imposes continuity of flux \eqref{inter2}, where $\mb{e}_{\Gamma}= I\otimes \mb{e}_1$, $\mb{d}_{\Gamma}^{T}\mb{w}$ and $\mb{u}_{\Gamma y}$ are the normal derivatives of the FD and DG solution, respectively.  The last two SAT in \eqref{semiFD} impose continuity of solution, where the parameters $\tauw$ and $\sigmaw$ will be chosen in the stability analysis such that a discrete energy estimate is obtained. 

Next, we consider the discretization in $\Omega_2$ by the IPDG method. Let $\mathcal{K}=\{K\}$ be a shape regular triangulation of $\Omega_2=\cup_{K\in\mathcal{K}}K$, and  $h_K$ be the diameter of $K$. The intersection between two adjacent triangles $K^+$, $K^-$ can be an edge, a point or empty. If $K^+\cap K^-$ is an edge, we call it an interior face, and denote $\mathcal{F}^I$ as the set of all interior faces. Similarly, we denote $\mathcal{F}^{\Gamma}$ the set of all interface faces $K\cap \Gamma$, and  $\mathcal{F}^{B}$ the set of all boundary faces $K\cap (\Omega_2/\Gamma)$. 

We define the function space 
\[
V_q = \{  v: v|_K\in\mathcal{P}^q(K), \forall K\in \mathcal{K}\},
\]
where $\mathcal{P}_q(K)$ is the space of polynomials of at most degree $q$ on $K$. 
The semidiscretization reads: for each fixed $t\in (0,T]$, find $u\in V_q$ such that
\begin{equation}\label{semiDG}
\begin{split}
(u_{tt}, \phi)_{\Omega_2} = -(b_2 \nabla u, \nabla \phi)_{\Omega_2} &+ \frac{1}{2}b_2\sum_{F\in \mathcal{F}^{\Gamma}} (u-w_F, \nabla\phi\cdot\mb{n})_{F}+\frac{1}{2}(b_2\nabla u\cdot\mb{n}-b_1 w_F^{\nabla}, \phi)_{F}\\
&-\frac{\tauu}{h}b_2\sum_{F\in \mathcal{F}^{\Gamma}}  (u-{w}_{F}, \phi)_{F}-\frac{\sigmau}{h}b_1(\tilde{\tilde u}-\tilde{w}, \phi)_{F},
\end{split}
\end{equation}
for all $\phi\in V_q$, where 
$w_F = \sum_{j=1}^{q+1}(P_{w2u}^g \mb{w}_{\Gamma})_j^F\phi_j^F$, $w_F^{\nabla}=  \sum_{j=1}^{q+1}(P_{w2u}^b\mb{d}_{\Gamma}^T\mb{w})_j^F\phi_j^F$, $\tilde{\tilde u}  =  \sum_{j=1}^{q+1}(P_{w2u}^b P_{u2w}^g \mb{u}_{\Gamma})_j^F\phi_j^F $, and $\tilde w  =  \sum_{j=1}^{q+1}(P_{w2u}^b  \mb{w}_{\Gamma})_j^F\phi_j^F $. On the right-hand side of \eqref{semiDG}, the first term is obtained by Green's identity. The second, fourth and fifth term impose continuity of solution. The third term imposes continuity of flux and takes into account the boundary term from using Green's identity. The  parameters
 $\tauu$ and $\sigmau$ are determined through stability analysis.
 
\begin{theorem}
The semidiscretization \eqref{semiFD}-\eqref{semiDG} is  stable if $\tauw=\sigmau\geq \frac{1}{4\beta}$ and $\tauu=\sigmaw\geq\frac{1}{C_{tr}}$, where the constant $\beta$ and $C_{tr}$ are independent of the mesh sizes. 
\end{theorem}
 
\begin{proof}
Multiplying \eqref{semiFD} by $\mb{w}_{t}^T\mathcal{H}$, we obtain
\begin{align*}
\mb{w}_{t}^T\mathcal{H}\mb{w}_{tt} = b_1 \mb{w}_t^T \mathcal{HD}\mb{w} + \frac{1}{2}b_1 \mb{w}_t^T \mb{d}_{\Gamma} H (\mb{w}_{\Gamma}-P_{u2w}^{g}\mb{u}_{\Gamma}) -   \frac{1}{2} \mb{w}_t^T \mb{e}_{\Gamma} H (b_1\mb{d}_{\Gamma}^T \mb{w}+P_{u2w}^b b_2 \mb{u}_{\Gamma y})\\
-b_1\frac{\tau_w}{h_1}\mb{w}_t^T \mb{e}_{\Gamma} H (\mb{w}_{\Gamma}-P_{u2w}^g\mb{u}_{\Gamma}) - b_2\frac{\sigma_w}{h_2} \mb{w}_t^T\mb{e}_{\Gamma}HP_{u2w}^b(P_{w2u}^g \mb{w}_{\Gamma}-u_{\Gamma}).
\end{align*}
On the right-hand side, we use the SBP property \eqref{D2} in the first term and obtain 
\begin{equation}\label{HDtemp}
b_1 \mb{w}_t^T \mathcal{HD}\mb{w} = b_1 \mb{w}_t^T(-A\otimes H - H\otimes A +  \mb{e}_{\Gamma}  H  \mb{d}_{\Gamma}^T)\mb{w}.
\end{equation}
In the above, we have only included the boundary term corresponding to the FD-DG interface. For the four SAT, we use the norm compatibility property \eqref{nc_fddg} to eliminate all projection operators with superscript $b$. After combining with \eqref{HDtemp}, we have 
\small
\begin{align}\label{energy_FD}
\begin{split}
\mb{w}_{t}^T\mathcal{H}\mb{w}_{tt} =& b_1\mb{w}_{t}^T (-A\otimes H - H\otimes A +  \mb{e}_{\Gamma}  H  \mb{d}_{\Gamma}^T)\mb{w} \\
&+\frac{1}{2}b_1 \mb{w}_t^T \mb{d}_{\Gamma} H \mb{w}_{\Gamma}-\frac{1}{2}b_1 \mb{w}_t^T \mb{d}_{\Gamma} H P_{u2w}^{g}\mb{u}_{\Gamma} -\frac{1}{2}b_1 \mb{w}_t^T \mb{e}^{\Gamma} H \mb{d}_{\Gamma}^T \mb{w} -\frac{1}{2}b_2 \mb{w}_t^T \mb{e}_{\Gamma} (M P_{w2u}^g)^T \mb{u}_{\Gamma y} \\
&-b_1\frac{\tau_w}{h_1}\mb{w}_t^T \mb{e}_{\Gamma} H \mb{w}_{\Gamma}+b_1\frac{\tau_w}{h_1}\mb{w}_t^T \mb{e}_{\Gamma} H  P_{u2w}^g\mb{u}_{\Gamma}- b_2\frac{\sigma_w}{h_2} \mb{w}_t^T\mb{e}_{\Gamma}(MP_{w2u}^g)^T P_{w2u}^g \mb{w}_{\Gamma}+b_2\frac{\sigma_w}{h_2} \mb{w}_t^T\mb{e}_{\Gamma}(MP_{w2u}^g)^T \mb{u}_{\Gamma}\\
=&\frac{d}{dt}\left[  -\frac{1}{2}b_1\mb{w}^T (A\otimes H + H\otimes A) \mb{w} +  \frac{1}{2}b_1\mb{w}_{\Gamma}^T H \mb{d}_{\Gamma}^T \mb{w}-\frac{1}{2}b_1\frac{\tauw}{h_1} \mb{w}_{\Gamma}^T H \mb{w}_{\Gamma} - \frac{1}{2}b_2\frac{\sigmaw}{h_2} (P_{w2u}^g \mb{w}_{\Gamma})^T M (P_{w2u}^g \mb{w}_{\Gamma}) \right]\\
&-\frac{1}{2}b_1 (\mb{d}_{\Gamma}^T \mb{w}_t)^T H P_{u2w}^g \mb{u}_{\Gamma} -\frac{1}{2}b_2 (P_{w2u}^g\mb{e}_{\Gamma}^T \mb{w}_t)^T M \mb{u}_{\Gamma y} +b_1\frac{\tauw}{h_1} (\mb{e}_{\Gamma}^T \mb{w}_t)^T H P_{u2w}^g \mb{u}_{\Gamma} + b_2\frac{\sigmaw}{h_2} (P_{w2u}^g \mb{e}_{\Gamma}^T \mb{w}_t)^T M \mb{u}_{\Gamma}. 
\end{split}
\end{align}
\normalsize

Next, we consider the DG discretization \eqref{semiDG}. By choosing $\phi=u_t$, we have 
\begin{equation}\label{semiDG2}
\begin{split}
(u_{tt}, u_t)_{\Omega_2} = -(b_2 \nabla u, \nabla u_t)_{\Omega_2} &+ \frac{1}{2}b_2\sum_{F\in \mathcal{F}^{\Gamma}} (u-w_F, \nabla u_t \cdot\mb{n})_{F}+\frac{1}{2}(b_2\nabla u\cdot\mb{n}-b_1 w_F^{\nabla}, u_t)_{F}\\
&-\frac{\tauu}{h}b_2\sum_{F\in \mathcal{F}^{\Gamma}}  (u-{w}_{F}, u_t)_{F}-\frac{\sigmau}{h}b_1(\tilde{\tilde u}-\tilde{w}, u_t)_{F},
\end{split}
\end{equation}
For the numerical fluxes on the right-hand side of \eqref{semiDG2}, we write in a matrix form as $\sum_{F\in \mathcal{F}^{\Gamma}} (p,q)_F=\mb{p}^T M\mb{q}$, and use the norm compatibility property \eqref{nc_fddg} to eliminate all projection operators with superscript $b$,
\begin{equation}\label{energy_DG}
\begin{split}
(u_{tt}, u_t)_{\Omega}=&\frac{d}{dt}\left[ -\frac{1}{2}b_2 \|\nabla u\|^2_{\Omega} + \frac{1}{2}b_2 \mb{u}_{\Gamma}^T M \mb{u}_{\Gamma y}-\frac{1}{2}b_2\frac{\tauu}{h_2} \mb{u}_{\Gamma}^T M \mb{u}_{\Gamma} -\frac{1}{2}b_1\frac{\sigmau}{h_1} (P_{u2w}^g\mb{u}_{\Gamma})^T H P_{u2w}^g\mb{u}_{\Gamma}   \right] \\
&-\frac{1}{2}b_2 (P_{w2u}^g \mb{w}_{\Gamma})^T M (\mb{u}_{\Gamma y})_t - \frac{1}{2} b_1 ( \mb{d}_{\Gamma}^T\mb{w})^T H P_{u2w}^g  (\mb{u}_{\Gamma })_t \\
&+ b_2\frac{\tauu}{h_2} (P_{w2u}^g \mb{w}_{\Gamma})^T M (\mb{u}_{\Gamma })_t + b_1 \frac{\sigmau}{h_1}(\mb{w}_{\Gamma})^T HP_{u2w}^g (\mb{u}_{\Gamma})_t.
\end{split}
\end{equation}

We collect the mixed terms from \eqref{energy_FD} and \eqref{energy_DG}, i.e. the terms that are not inside the square-brackets for the time derivative $d/dt$. By requiring $\sigmaw=\tauu$ and $\sigmau=\tauw$, we can write all mixed terms in the form of the time derivative,
\begin{equation}\label{mixed}
\begin{split}
\frac{d}{dt}\left[ -\frac{1}{2}b_1 (\mb{d}_{\Gamma}^T\mb{w})^THP_{u2w}^g \mb{u}_{\Gamma}-\frac{1}{2}b_2 (P_{w2u}^g\mb{w}_{\Gamma})^T M\mb{u}_{\Gamma y} +b_1\frac{\tauw}{h_1}\mb{w}_{\Gamma}^T H P_{u2w}^g \mb{u}_{\Gamma} + b_2\frac{\tauu}{h_2} (P_{w2u}^g \mb{w}_{\Gamma})^T M \mb{u}_{\Gamma} \right].
\end{split}
\end{equation}
Combining \eqref{energy_FD}, \eqref{energy_DG} and \eqref{mixed}, we obtain 
\[
\frac{1}{2}\frac{d}{dt}\left(\mb{w}_t^T \mathcal{H} \mb{w} + (u_t, u_t)_{\Omega}\right)=\frac{d}{dt}(\tilde E_1+\tilde E_2),
\]
where 
\begin{align*}
\tilde E_1=& -\frac{1}{2}b_1 \mb{w}^T (A\otimes H+H\otimes A)\mb{w} +  \frac{1}{2}b_1 \mb{w}_{\Gamma} H \mb{d}_{\Gamma}^T \mb{w} -\frac{1}{2}b_1\frac{\tauw}{h_1} \mb{w}_{\Gamma}^T H \mb{w}_{\Gamma}\\
&-\frac{1}{2}b_1
\frac{\tauw}{h_1} (P_{u2w}\mb{u}_{\Gamma})^TH(P_{u2w}\mb{u}_{\Gamma})-\frac{1}{2}b_1 (\mb{d}_{\Gamma}^T \mb{w})^TH(P_{u2w}^g\mb{u}_{\Gamma})+b_1\frac{\tauw}{h_1}(\mb{w}_{\Gamma})^T H(P_{u2w}^g\mb{u}_{\Gamma}),\\
\tilde E_2=& -\frac{1}{2}b_2 \|\nabla u\|^2_{\Omega} + \frac{1}{2}b_2 \mb{u}_{\Gamma} M \mb{u}_{\Gamma y}  -\frac{1}{2}b_2\frac{\tauu}{h_2}\mb{u}_{\Gamma}^T M \mb{u}_{\Gamma} \\
&-\frac{1}{2}b_2\frac{\tauu}{h_2}(P_{w2u}\mb{w}_{\Gamma})^TM(P_{w2u}\mb{w}_{\Gamma})-\frac{1}{2}b_2(P_{w2u}^g\mb{w}_{\Gamma})^T M \mb{u}_{\Gamma y}+b_2\frac{\tauu}{h_2} (P_{w2u}^g \mb{w}_{\Gamma})^T M \mb{u}_{\Gamma}.
\end{align*}

In the following, we determine the penalty parameters such that $\tilde E_1\leq 0$ and $\tilde E_2\leq 0$. We start with $\tilde E_1$, and write 
\small
\begin{align*}
\tilde E_1=& -\frac{1}{2}b_1 \mb{w}^T (A\otimes H+H\otimes A)\mb{w} -\frac{1}{2}b_1\frac{\tauw}{h_1}(\mb{w}_{\Gamma}-P_{u2w}^g \mb{u}_{\Gamma})^T H (\mb{w}_{\Gamma}-P_{u2w}^g \mb{u}_{\Gamma}) + \frac{1}{2}b_1 (\mb{d}_{\Gamma}^T\mb{w})^T H(\mb{w}_{\Gamma}-P_{u2w}^g\mb{u}_{\Gamma})\\
=& -\frac{1}{2}b_1 \mb{w}^T (A\otimes H+H\otimes A)\mb{w} \\
&-\frac{1}{2}b_1\frac{\tauw}{h_1}(\mb{w}_{\Gamma}-P_{u2w}^g \mb{u}_{\Gamma}-\frac{h_1}{2\tauw}\mb{d}_{\Gamma}^T\mb{w})^T H (\mb{w}_{\Gamma}-P_{u2w}^g \mb{u}_{\Gamma}-\frac{h_1}{2\tauw}\mb{d}_{\Gamma}^T\mb{w}) +b_1\frac{h_1}{8\tauw}(\mb{d}_{\Gamma}^T\mb{w})^TH(\mb{d}_{\Gamma}^T\mb{w})
\end{align*}
\normalsize
To control the last term on the right-hand side, we use the borrowing technique,
\[
\mb{w}^T (H\otimes A) \mb{w}=\mb{w}^T (H\otimes \tilde A) \mb{w}+\mb{w}^T (H\otimes (\beta h_1\mb{d}_1\mb{d}_1^T))\mb{w}=\mb{w}^T (H\otimes \tilde A) \mb{w}+\beta h_1   (\mb{d}_{\Gamma}^T\mb{w})^TH(\mb{d}_{\Gamma}^T\mb{w}).
\]
To guarantee $\tilde E_1\leq 0$, we require
\[
-\frac{1}{2}b_1\beta h_1+b_1\frac{h_1}{8\tauw}\leq 0\Rightarrow \tauw\geq\frac{1}{4\beta}.
\]
Similarly, by using the trace inequality $\|\nabla u\cdot\mb{n}\|_{\Gamma}^2\leq C_{tr}h_2^{-1} \|\nabla u\|_{\Omega}^2$, the condition $\tauu\geq\frac{1}{4C_{tr}}$ guarantees that $\tilde E_2\leq 0$. This concludes the proof. 
\end{proof}

\subsection{Truncation error}\label{sec_ta}
To preserve the convergence rate $p+2$, the truncation error of the FD penalty terms and DG fluxes must be small enough. To be precise, we shall distinguish DOFs that are in the interior of the interface from DOFs near the edges, because of different projection errors. By using the accuracy properties of the SBP operators and the projection operators, we analyze the truncation error of each SAT in the FD semidiscretization and of each numerical flux in the DG semidiscretization, see the result in Table \ref{tab_truncationerror}. 

\begin{table}[]
    \centering
    \begin{tabular}{ccc}
        SAT/Flux & $Interior$ & $Edge$     \\
        \hline 
        1 & $2p-2$ & $p-1$ \\
        2 & $p$ & $p-1$  \\
        3 & $2p-2$ & $p-1$  \\
        4 & $2p-2$   &$p-1$   \\
            \hline
    \end{tabular}
    \caption{Truncation errors in the SAT on the FD side and the numerical fluxes on the DG side}
    \label{tab_truncationerror}
\end{table}

In the interior of the interface, the projection error is $\mathcal{O}(h^{2p})$. The weights in the first, third and fourth SAT/flux include a factor $h^{-2}$, thus resulting a truncation error  $\mathcal{O}(h^{2p-2})$. In the second SAT/flux, however, the truncation error is dictated by the first derivative approximation of order $p+1$, when combined with a weight of $h^{-1}$, the truncation error is $\mathcal{O}(h^{p})$. Since the number of DOFs in the interior of the interface is $\mathcal{O}(h^{-1})$, we expect a gain of two order in convergence rate, i.e. $\min(2p,p+2)$. 

Next, we consider the truncation error on a few grid points near the edges of the interface, where the projection error is $\mathcal{O}(h^{p+1})$ for the projection operators with superscript $g$ and $\mathcal{O}(h^{p})$ for the projection operators with superscript $b$. The same calculation shows that all the four SAT/flux have a truncation error $\mathcal{O}(h^{p-1})$. Thus, a gain of three orders is needed for a convergence rate $p+2$, which can be expected because the number of grid points with truncation error $\mathcal{O}(h^{p-1})$ is $\mathcal{O}(1)$ and the total number of grid points is $\mathcal{O}(h^{-2})$. The theoretical analysis is out of the scope of this work, but a gain of three orders for a simplified model problem was analyzed in \cite{Wang2018b}. Indeed, we observe a convergence rate $p+2$ in the numerical experiments for the case $p=2$.

\section{Numerical experiments}
In this section, we present numerical experiments for the FD-DG discretization for the 2D wave equation. We start with a verification of the convergence rate by using an example based on Snell's law. After that, to demonstrate robustness we consider an example with a complex geometry cannot be easily resolved by using a curvilinear grid technique. 

\subsection{Verification of convergence rate}
Consider the wave equation 
\begin{align}
U_{tt}&=\nabla \cdot b_1\nabla U,\quad \text{ in } \Omega_1,\label{Ub1}\\
U_{tt}&=\nabla \cdot b_2\nabla U,\quad \text{ in } \Omega_2,\label{Ub2}
\end{align}
with piecewise constant material property $b_1=1$ and $b_2=0.25$, and $\Omega_1=[0,10]^2$, $\Omega_2=[0,10]\times[-2,0]$. At the material interface $y=0$, we prescribe continuity of solution and flux \eqref{inter1}-\eqref{inter2}. To close the problem, we impose Dirichlet boundary conditions at all boundaries. By using Snell's law, an analytical solution takes the form 
\begin{align*}
U=\begin{cases}
\cos(x+y-\sqrt{2b_1}t) + k_2\cos(x-y+\sqrt{2b_1}t), \quad \text{ in }\Omega_1,\\
(1+k_2)\cos(k_1x+y-\sqrt{2b_1}t), \quad \text{ in }\Omega_2,
\end{cases}
\end{align*}
where $k_1 = \sqrt{2b_1/b_2-1}$ and $k_2 = (b_1-k_1b_2)/(b_1+k_1b_2)$. We plot the solution at $t=0$ in Figure \ref{SnellInitialData}, and observe the shorter wavelength in $\Omega_2$ because of the slower wave speed. In the numerical experiments, we use this analytical solution to obtain initial and boundary data. 

For spatial discretization, we use the SBP operators with order of accuracy (4,2) on a Cartesian grid in $\Omega_1$. The IPDG method with local polynomials of degree three is used in $\Omega_2$. We consider two different mesh configurations in $\Omega_2$. First, for the structured mesh in Figure \ref{SnellMesh3} and the unstructured mesh in \ref{SnellMesh1},  the vertices of the triangles on the interface coincide with the grid points on the FD side. Since there are ten DOFs in each DG element and four DOFs on each triangle edge, the DOFs on the interface from the two discretizations do not coincide. Second, for the structured mesh in Figure \ref{SnellMesh4} and the unstructured mesh in Figure \ref{SnellMesh2}, the vertices of the triangles on the interface coincide with every third grid point on the FD side. When positioning the DG DOFs on the principal lattice, the Lagrange nodes on the interface coincide with the FD grid points on the interface. We note that even in this case, the projection operators in the semidiscretization cannot be identity, because the discrete FD and DG norms are not the same on the interface. 

\begin{figure}
 \begin{subfigure}[b]{0.19\textwidth}
\centering
\includegraphics[trim={4cm 0 4cm 0},clip,width=0.97\textwidth]{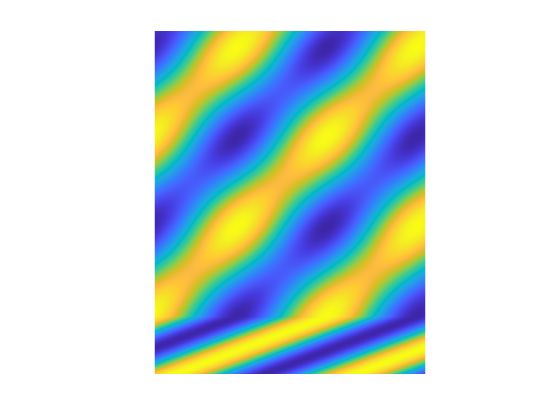}
\caption{}
\label{SnellInitialData}
     \end{subfigure}
\begin{subfigure}[b]{0.19\textwidth}
\centering
\includegraphics[trim={4cm 0 4cm 0},clip,width=0.99\textwidth]{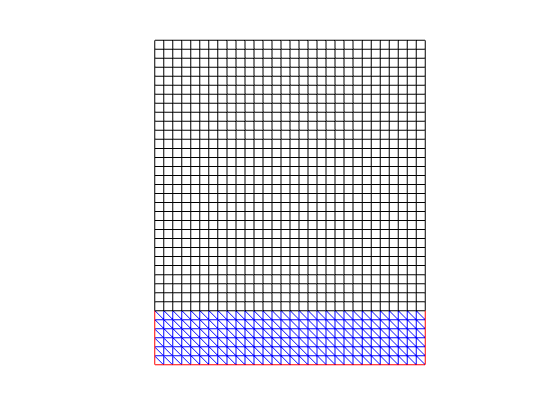}
     \caption{}
     \label{SnellMesh3}
\end{subfigure}
      \begin{subfigure}[b]{0.19\textwidth}
      \centering
\includegraphics[trim={4cm 0 4cm 0},clip,width=0.99\textwidth]{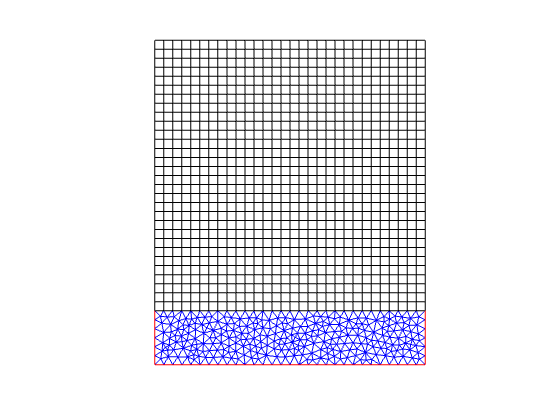}
\caption{}
\label{SnellMesh1}
     \end{subfigure}
\begin{subfigure}[b]{0.19\textwidth}
\centering
\includegraphics[trim={4cm 0 4cm 0},clip,width=0.99\textwidth]{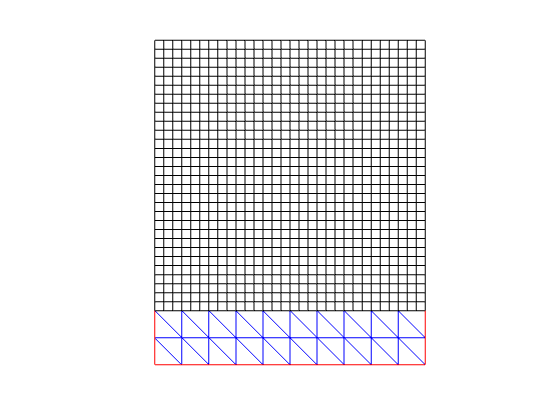}
     \caption{}
     \label{SnellMesh4}
\end{subfigure}
\begin{subfigure}[b]{0.19\textwidth}
\centering
\includegraphics[trim={4cm 0 4cm 0},clip,width=0.99\textwidth]{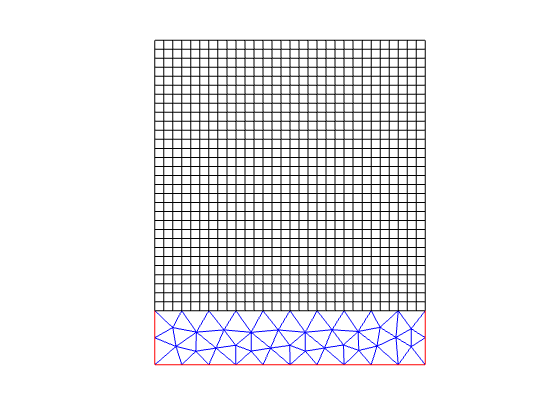}
     \caption{}
     \label{SnellMesh2}
\end{subfigure}
\caption{(a) The Snell's law solution at $t=0$. (b-e) Meshes}
\end{figure}

In the semidiscretization \eqref{semiFD}-\eqref{semiDG} as a system of second order ordinary differential equations (ODEs), we discretize the time variable by the fourth order accurate modified equation method \cite{Gilbert2008}. We choose the final time $T=2$,  and a time step small enough so that the error in the solution is dominated by the spatial approximation. In Table \ref{tab_convergence}, we present the $l^2$ errors and the corresponding convergence rates for the four types of meshes in Figure \ref{SnellMesh3}-\ref{SnellMesh2}. We observe a fourth order convergence rate in all cases. 

\begin{table}[]
    \centering
    \begin{tabular}{lllll}
        $n$ & Error 1 (rates) & Error 2 (rates) & Error 3 (rates) & Error 4 (rates)     \\
        \hline 
        31 & $3.6864 \times 10^{-2}$              &   $3.5864 \times 10^{-2}$             & $9.2747\times 10^{-2}$ &          $1.7052\times 10^{-2}$ \\
        61 & $1.5993 \times 10^{-3} (4.53)$    &   $1.5946 \times 10^{-3} (4.49)$  & $5.4464\times 10^{-3}(4.09)$ & $1.6190\times 10^{-3}(3.40)$ \\
        121 &$0.9695\times 10^{-4} (4.04)$    &   $0.9027\times 10^{-4} (4.14)$   & $3.2015\times 10^{-4}(4.09)$ &$9.0579\times 10^{-5}(4.16)$ \\
        241 &$0.5644\times 10^{-5} (4.10)$    &   $0.5525\times 10^{-5} (4.03)$   & $2.0770\times 10^{-5}(3.95)$& $4.7953\times 10^{-6}(4.24)$ \\
            \hline
    \end{tabular}
    \caption{Convergence rates. Error 1-4 correspond to the mesh configurations in Figure \ref{SnellMesh3}-\ref{SnellMesh2}, respectively. The value $n$ is the number of grid points in $\Omega_1$ in each spatial direction.}
    \label{tab_convergence}
\end{table}

\subsection{Complex geometry}
In this numerical example, we consider a layered medium with complex geometry. The computational domain consists of two subdomains  $\Omega_1=[0,1]\times[0,1.5]$ and $\Omega_2=[0,1]\times[-0.5,0]$, with the same governing equation \eqref{Ub1}-\eqref{Ub2} and material properties as in the previous example, $b_1=1$ and $b_2=0.25$. In $\Omega_2$, there are three cavities of irregular shapes, see Figure \ref{ComplexGeometryMesh}. The complex geometry is very difficult to resolve by using curvilinear grids. Instead, we use an unstructured mesh and discretize the governing equation by the IPDG method. In $\Omega_1$, the SBP FD method on a Cartesian grid is used for the spatial discretization. 

At time $t=0$, we initialize a Gaussian profile $U(x,y,0)= 10e^{-1000((x-0.5)^2+(y+0.3)^2)}$ plotted in Figure \ref{ComplexGeometryt0}. We set the velocity to be zero and  impose homogeneous Dirichlet boundary conditions at the outer boundaries and the cavity boundaries. The solutions at $t=0.3, 0.5$ and 0.9 in Figure \ref{ComplexGeometryt3}-\ref{ComplexGeometryt9} show the wave interaction with the cavities. At $t=1.5$ in Figure \ref{ComplexGeometryt1.5}, the wave has passed the material interface to $\Omega_1$ and the wavelength becomes larger. In the last plot in Figure \ref{ComplexGeometryt10}, the wave has spread in the entire computational domain. 
It is clear that the waves are well-resolved and the number method is stable.
\begin{figure}
\centering 
\begin{subfigure}[b]{0.13\textwidth}
\includegraphics[width= .99\textwidth,trim={7cm 1cm 6.5cm 1cm},clip]{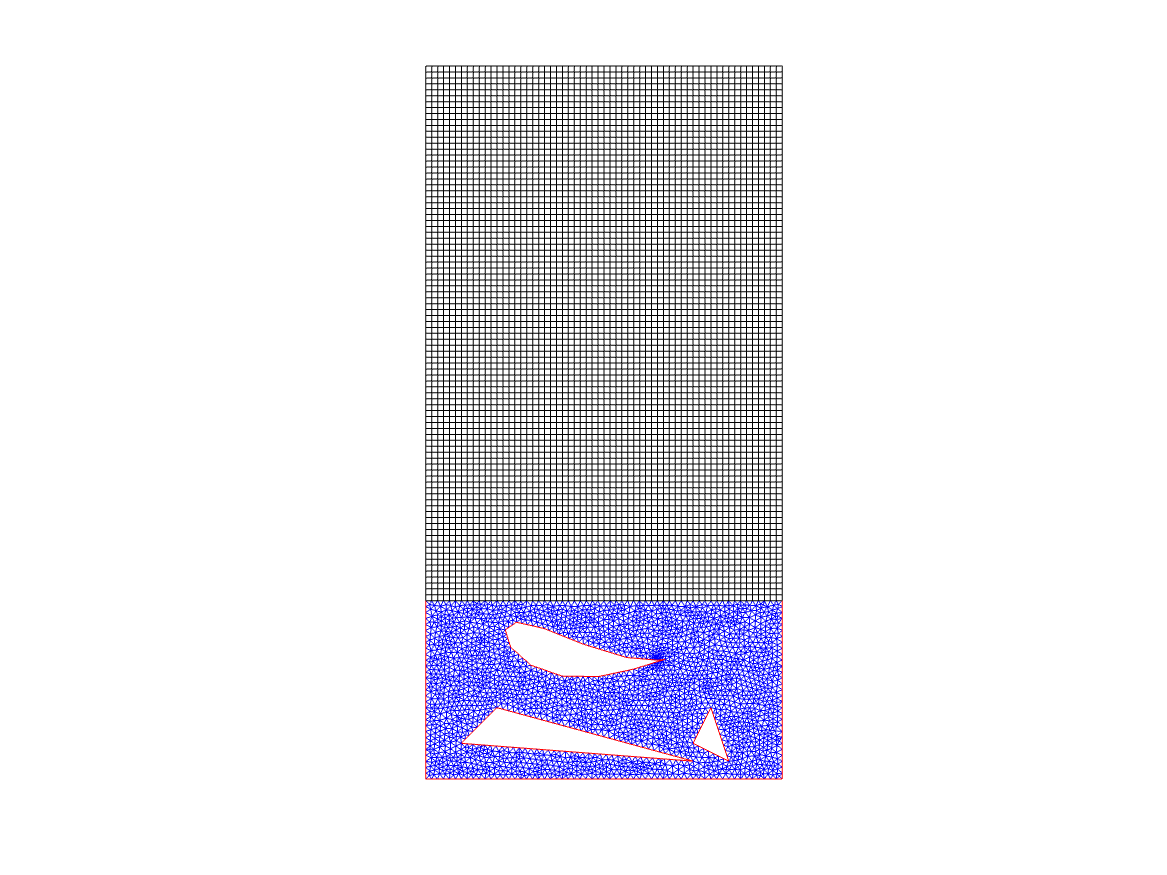}
         \caption{}
         \label{ComplexGeometryMesh}
     \end{subfigure}
\begin{subfigure}[b]{0.13\textwidth}
\includegraphics[width= .99\textwidth,trim={7cm 1cm 6.5cm 1cm},clip]{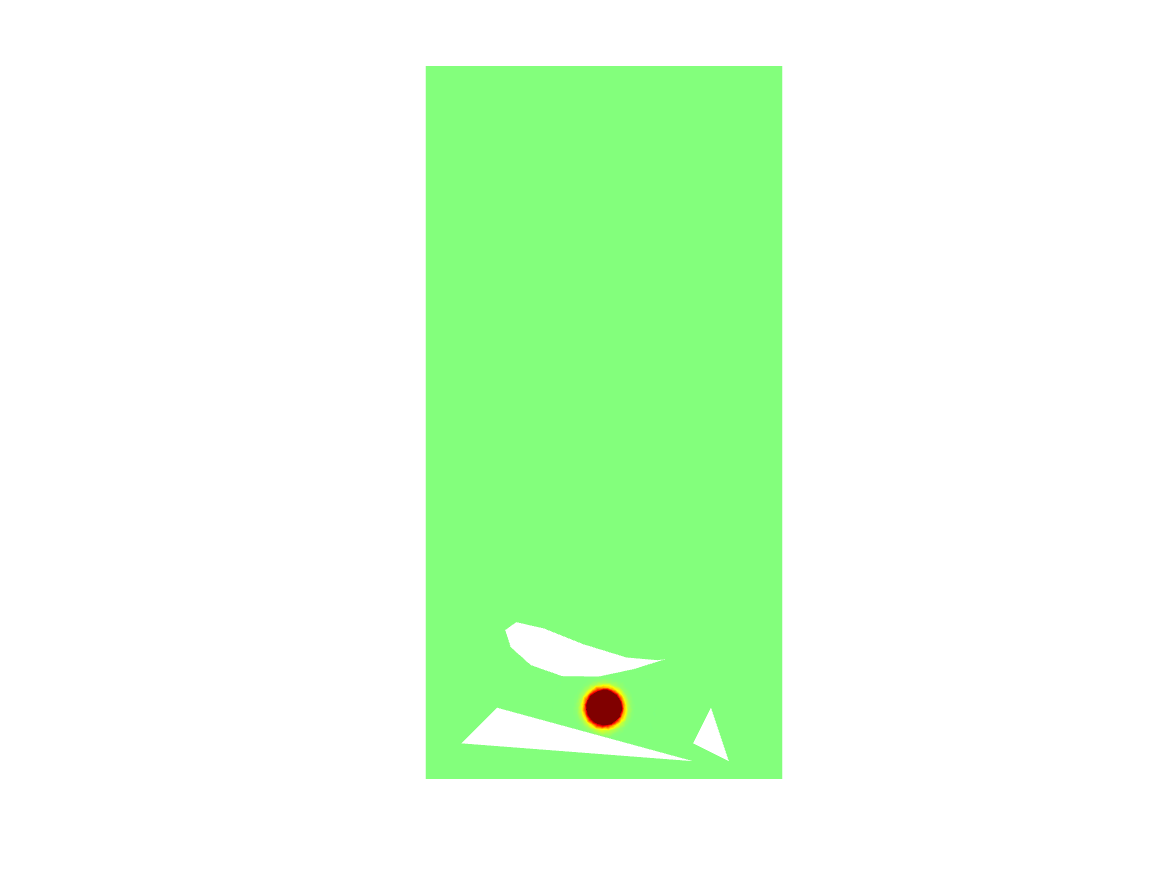}
         \caption{}
                  \label{ComplexGeometryt0}
     \end{subfigure}
\begin{subfigure}[b]{0.13\textwidth}
\includegraphics[width= .99\textwidth,trim={7cm 1cm 6.5cm 1cm},clip]{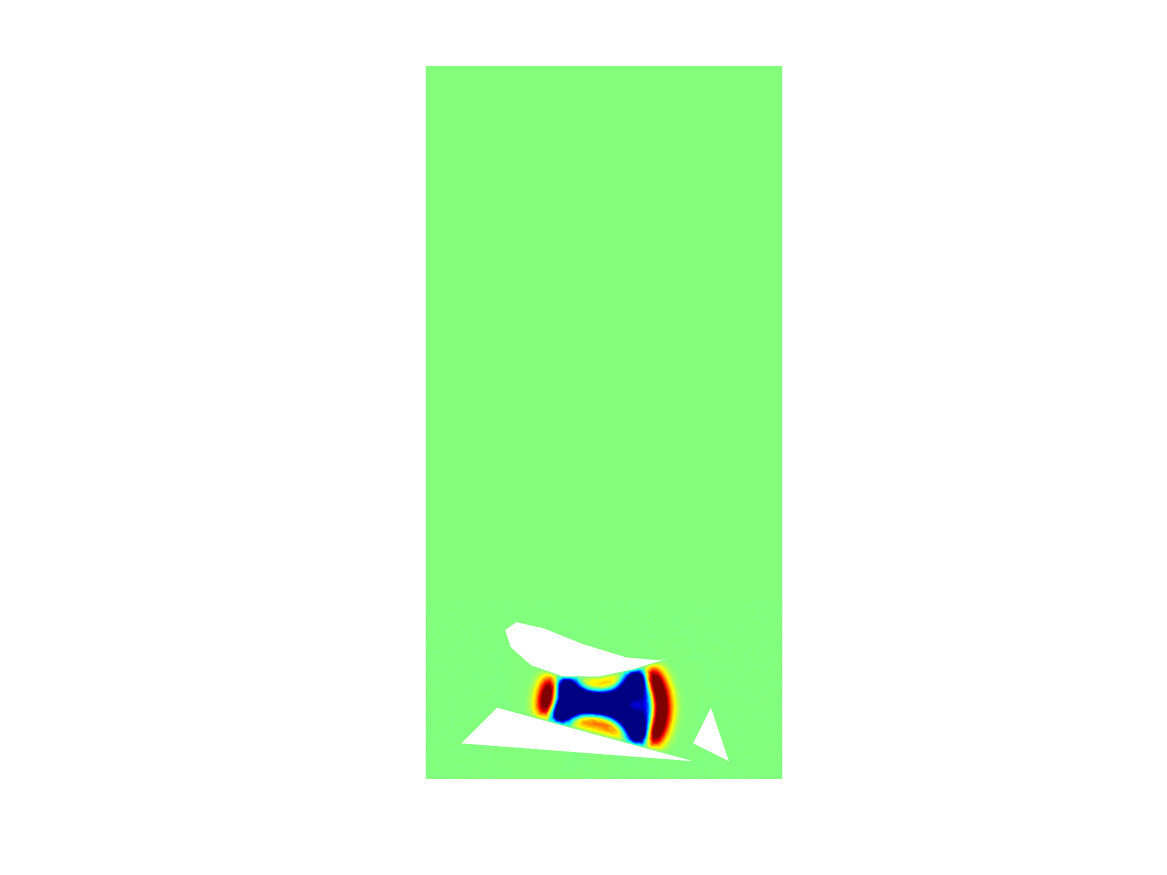}
         \caption{}
                           \label{ComplexGeometryt3}
     \end{subfigure}
\begin{subfigure}[b]{0.13\textwidth}
\includegraphics[width= .99\textwidth,trim={7cm 1cm 6.5cm 1cm},clip]{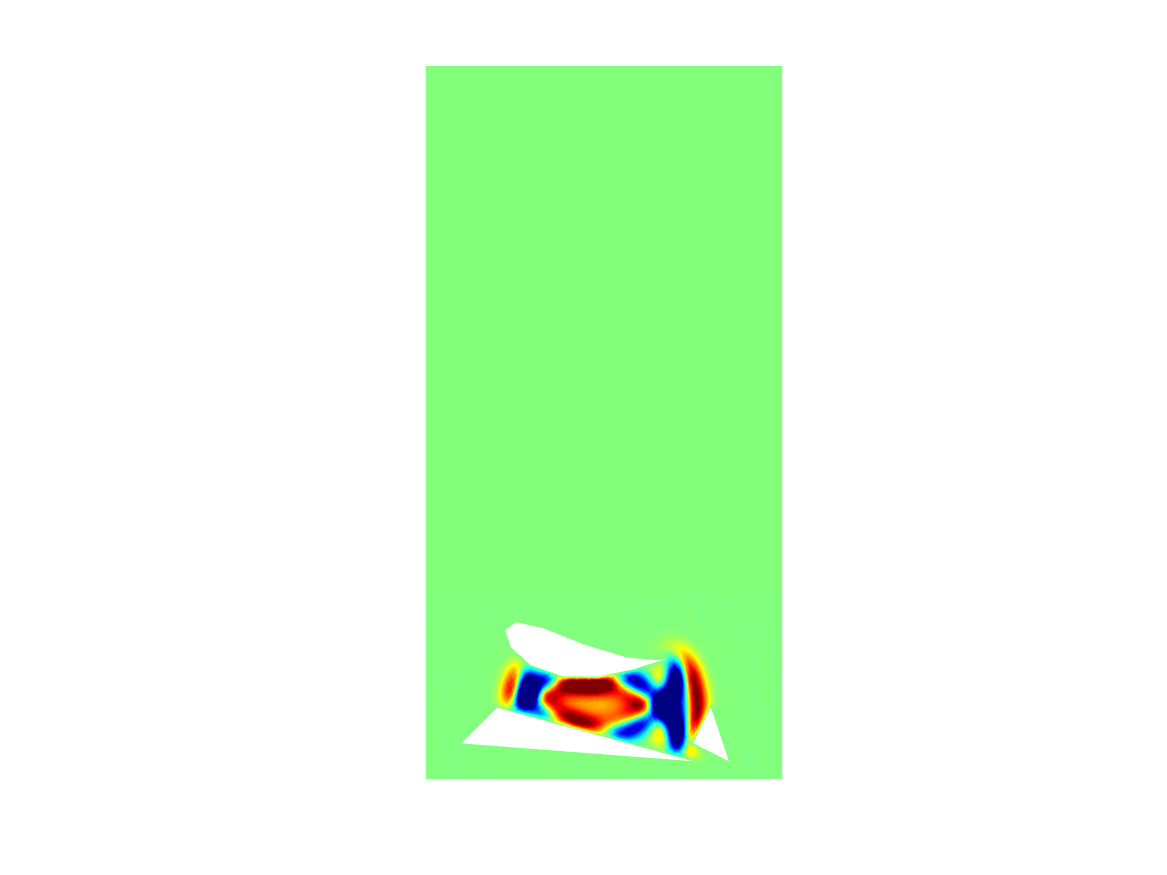}
         \caption{}
     \end{subfigure}
\begin{subfigure}[b]{0.13\textwidth}
\includegraphics[width= .99\textwidth,trim={7cm 1cm 6.5cm 1cm},clip]{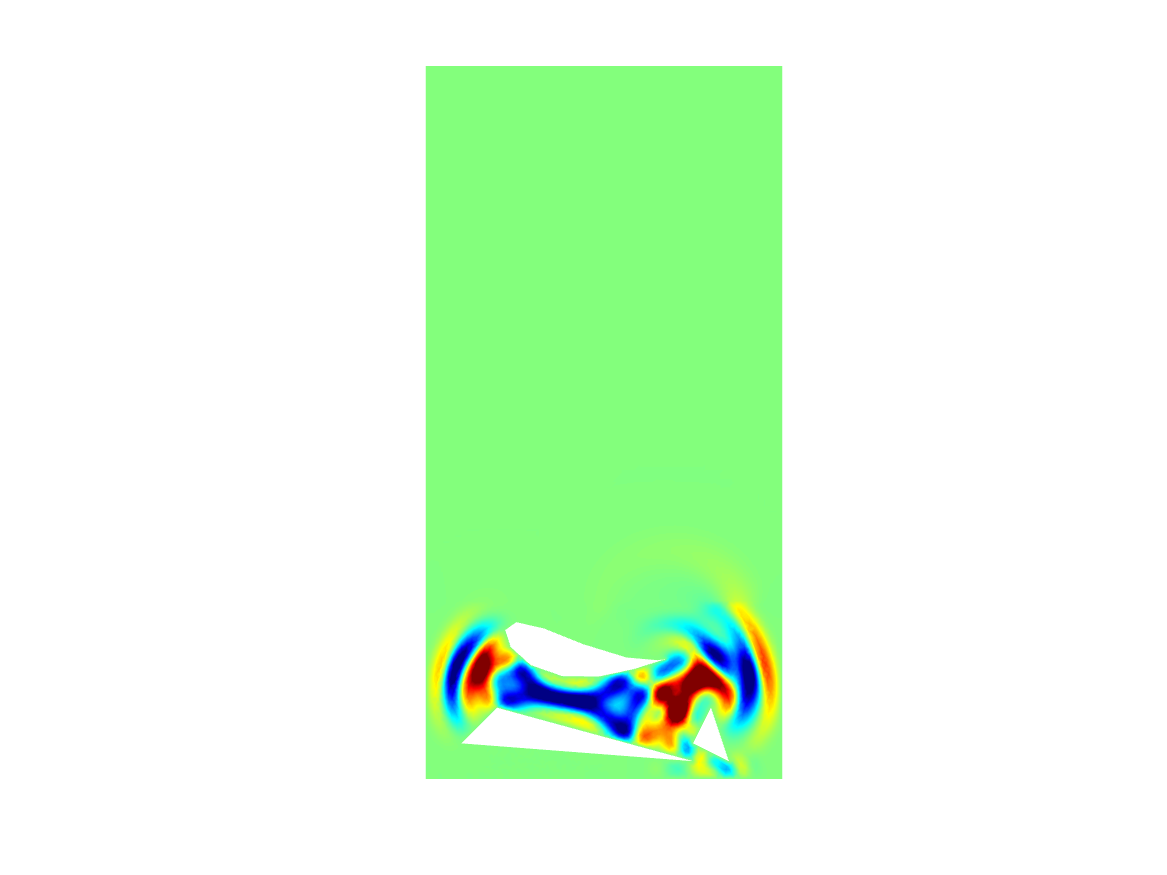}
         \caption{}                        
          \label{ComplexGeometryt9}
     \end{subfigure}
\begin{subfigure}[b]{0.13\textwidth}
\includegraphics[width= .99\textwidth,trim={7cm 1cm 6.5cm 1cm},clip]{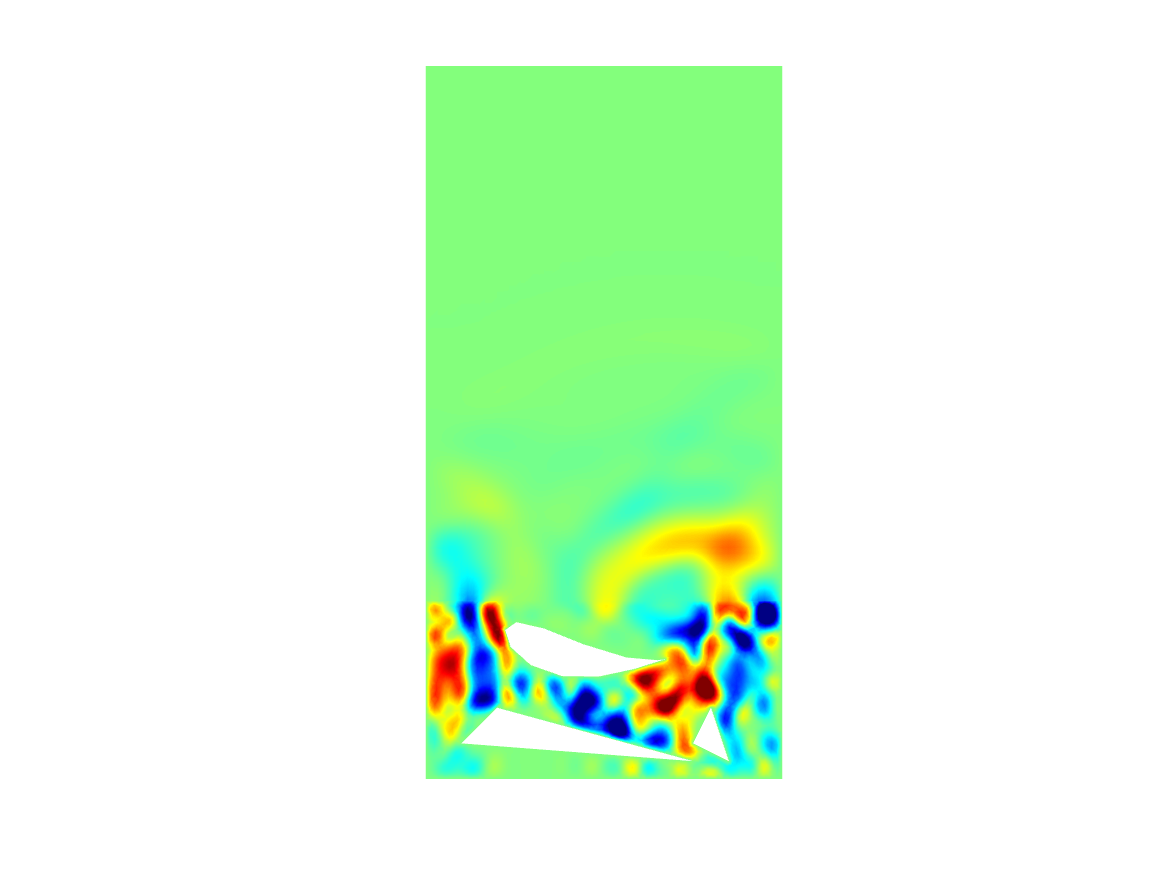}
         \caption{}
         \label{ComplexGeometryt1.5}
     \end{subfigure}
\begin{subfigure}[b]{0.13\textwidth}
\includegraphics[width= .99\textwidth,trim={7cm 1cm 6.5cm 1cm},clip]{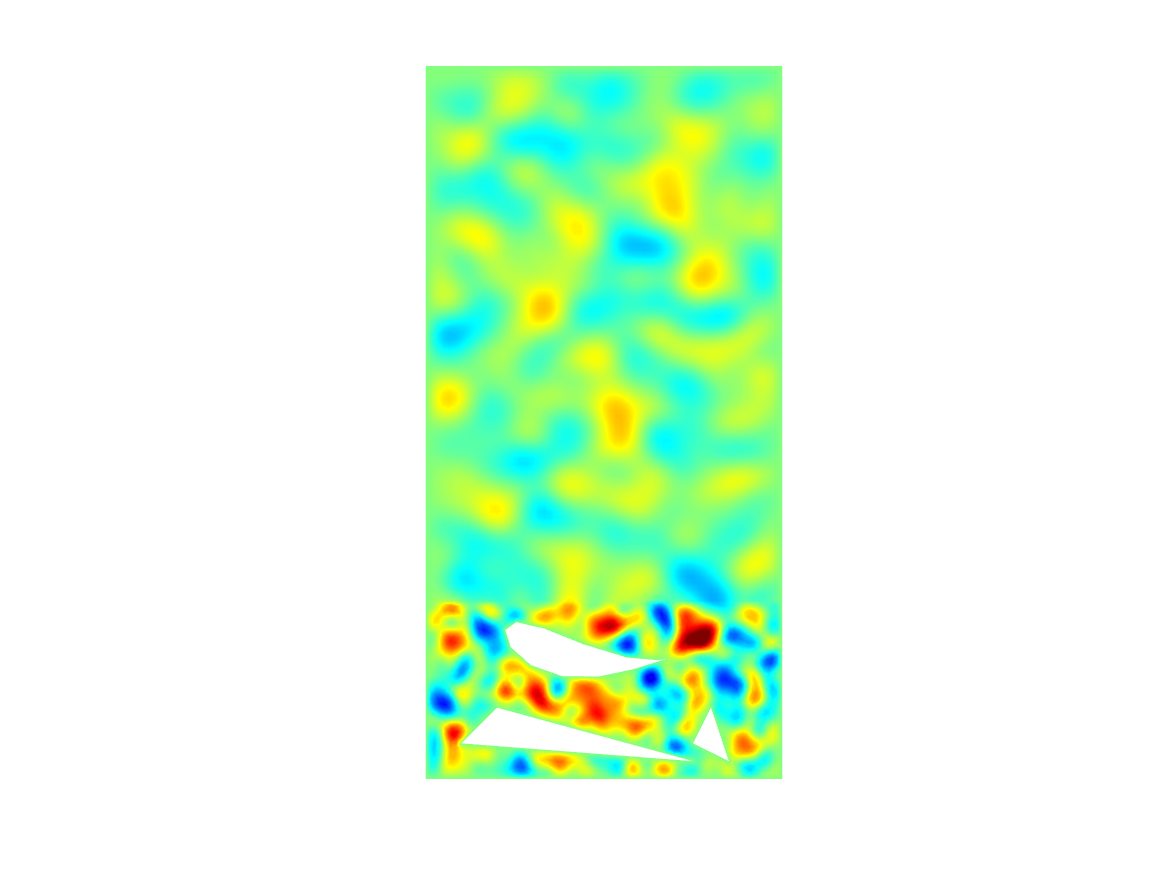}
         \caption{}
                  \label{ComplexGeometryt10}
     \end{subfigure}
\caption{(a): 2D domain with a Cartesian grid on the top layer and a triangulation in the bottom layer.  (b)-(g): numerical solution at increasing times $t=0, 0.3, 0.5, 0.9, 1.5, 10$. }
\label{HybridMesh}
\end{figure}

\section{Conclusion}
We have developed an FD-DG discretization for the wave equation in second order form in two space dimension. The FD discretization is based on SBP operators on Cartesian grids. In the region with complex geometry or heterogeneous material property, the IPDG discretization on structured or unstructured meshes are used. We use the penalty technique to couple the FD and DG solutions. For this, we have constructed projection operators to move between pointwise FD solutions and DG solutions in a  space of piecewise polynomials. The projection operators are compatible with respect to the discrete norms from the FD and DG side, resulting an energy estimate for the overall semidiscretization.  In addition, the convergence rate of the FD-DG discretization is optimal, in the sense that it is the same as when one method is used in the entire domain.  
The hybridization combines computational efficiency of high-order finite differences and geometric flexibility of the discontinuous Galerkin technique.

The second main contribution is a new framework for deriving error estimates for the FD-DG discretization. On the DG side, we use a non-traditional approach by realizing the weak form as difference stencils and compute the truncation errors. It is well-known that the order of truncation error of these difference stencils are lower than the expected convergence rate. By exploring an analogue of  the Galerkin orthogonality, we prove the sharp error bounds for the DG discretization away from the interface by the energy method. We then use the normal mode analysis for the accuracy property at the FD-DG interface. In the end, we combine the error estimates in the interior and close to the interface. 

The FD-DG discretization finds immediately applications in other second order hyperbolic PDEs. In a coming work, we will consider the elastic wave equation modeling seismic wave propagation. Additionally, we will investigate local time stepping techniques \cite{Grote2021,Rietmann2017} to address different time step restrictions from the FD and DG discretizations.   

\bibliographystyle{siamplain}
\bibliography{/Users/siyangwang/Documents/Research/Siyang_References}

\end{document}